\documentclass[12pt]{article}
\usepackage[centertags]{amsmath}
\usepackage{amsfonts}
\usepackage{amssymb}
\usepackage{amsthm}
\usepackage{amsmath,mathrsfs}
\usepackage{amsmath,amscd}
\title{\textbf{Giraud's Theorem and Categories of Representations}}
\author{Renaud Gauthier \footnote{rg.mathematics@gmail.com} \\ \\}
\theoremstyle{definition}

\newtheorem{Giraud}{Theorem}[section]
\newtheorem{Lemma}[Giraud]{Lemma}
\newtheorem{sheaf}[Giraud]{Proposition}
\newtheorem{abOp}{Definition}[section]
\newtheorem{cohabop}[abOp]{Definition}
\newtheorem{Onm}[abOp]{Definition}
\newtheorem{umbrella}[abOp]{Definition}
\newtheorem{asscat}[abOp]{Definition}
\newtheorem{Noteasscat}[abOp]{Note}
\newtheorem{RnCat}[abOp]{Lemma}
\newtheorem{equ}[abOp]{Theorem}
\newtheorem{Yoneda}[abOp]{Lemma}
\newtheorem{embed}[abOp]{Corollary}
\newtheorem{rep}[abOp]{Proposition}
\newtheorem{Correp}[abOp]{Corollary}
\newtheorem{ultthm}{Theorem}[section]
\newtheorem{unit}[ultthm]{Lemma}
\newtheorem{Esheaf}[ultthm]{Lemma}
\newtheorem{repunit}[ultthm]{Lemma}
\newtheorem{etaiso}[ultthm]{Lemma}
\newtheorem{preco}[ultthm]{Lemma}
\newtheorem{locsurj}[ultthm]{Lemma}
\newtheorem{presepi}[ultthm]{Lemma}
\newtheorem{Ladj}[ultthm]{Lemma}
\newtheorem{ai}[ultthm]{Lemma}
\DeclareMathOperator*{\rrarrop}{\rightrightarrows}
\DeclareMathOperator*{\colim}{\text{colim}}
\newcommand{\beq}{\begin{equation}}
\newcommand{\eeq}{\end{equation}}
\newcommand{\rarr}{\rightarrow}
\newcommand{\Rarr}{\Rightarrow}
\newcommand{\xrarr}{\xrightarrow}
\newcommand{\Lrarr}{\Longrightarrow}
\newcommand{\rrarr}{\rightrightarrows}
\newcommand{\rlarr}{\rightleftarrows}
\newcommand{\str}{\text{str}}
\newcommand{\Str}{\text{Str}}
\newcommand{\eset}{\emptyset}
\newcommand{\Ob}{\text{Ob}}
\newcommand{\Mod}{\text{Mod}}
\newcommand{\Mor}{\text{Mor}}
\newcommand{\Nat}{\text{Nat}}

\newcommand{\cF}{\mathcal{F}}
\newcommand{\cO}{\mathcal{O}}
\newcommand{\cC}{\mathcal{C}}

\newcommand{\cD}{\mathcal{D}}
\newcommand{\cE}{\mathcal{E}}
\newcommand{\cP}{\mathcal{P}}
\newcommand{\cR}{\mathcal{R}}
\newcommand{\cS}{\mathcal{S}}
\newcommand{\cU}{\mathcal{U}}

\newcommand{\Ed}{E^{\downarrow}}
\newcommand{\Fd}{F^{\downarrow}}
\newcommand{\Hom}{\text{Hom}}
\newcommand{\bland}{\textbf{bland}}
\newcommand{\RMod}{R\text{-Mod}}
\newcommand{\RNatMod}{R\text{-NatMod}}
\newcommand{\ROhC}{R\text{-}\cO h_C}
\newcommand{\ROh}{R\text{-}\cO h}
\newcommand{\RNatOh}{R\text{-Nat}\cO h}
\newcommand{\RModCop}{(R\text{-Mod})^{\mathcal{C}^{op}}}
\newcommand{\RNatModCop}{(\RNatMod)^{\cC^{op}}}

\newcommand{\sRNatModCop}{!\RNatModCop}
\newcommand{\SetCop}{\text{Sets}^{\cC^{op}}}
\newcommand{\epiplus}{\text{epi}\mathbf{+}}
\newcommand{\ShCJ}{\text{Sh}(\cC,J)}
\newcommand{\ShRCJ}{\text{Sh}_{R}(\cC,J)}
\newcommand{\Cnm}{\cC^{(n|m)}}
\newcommand{\Oenm}{\cO^{(n|m)}}
\newcommand{\Pqr}{\cP^{(q|r)}}
\newcommand{\Skl}{\cS^{(k|l)}}
\begin{document}
\maketitle
\begin{abstract}
We present an alternate proof of Giraud's Theorem based on the fact that given the conditions on a category $\cE$ for being a topos, its objects are sheaves by construction. Generalizing sets to $R$-modules for $R$ a commutative ring, we prove that a category with small hom-sets and finite limits is equivalent to a category of sheaves of $R$-modules on a site if and only if it satisfies Giraud's axioms and in addition is enriched in a certain symmetric monoidal category parametrized by an $R$-module.
\end{abstract}

\newpage

\section{Introduction}
Giraud gave a characterization of topos that is independent of a site (\cite{G}). He gave conditions on a category $\cE$ to be a Grothendieck topos. The construction of such a category aims at first proving that Hom$_{\cE}(-, E)$ is a sheaf for all objects $E$ of $\cE$ and then uses a Hom-tensor adjunction to prove that $\cE$ is indeed equivalent to a Grothendieck topos. In an effort to try to understand all of Giraud's conditions abstractly rather than being a list of conditions to be met, we show that objects of $\cE$ are sheaves, which leads to an alternate proof of Giraud's Theorem, which is presented first. One question that comes to mind is whether one can generalize such a result to the case of categories of representations, functor categories of the form $\RModCop$ for a small category $\cC$, $R$ a fixed commutative ring with a unit. We show that a category $\cE$ with small hom-sets and a small set of generators is equivalent to a certain category of representations. However those latter categories are built with very restrictive conditions. In general one may ask under what conditions a category $\cE$ is equivalent to a category of representations $\RModCop$ for any commutative ring $R$, and this is proved in the last part of the paper, where in particular we show that to have such a result forces us to regard $\cE$ as a Grothendieck topos, hence to also put a Grothendieck topology on $\cC$, which ``topologizes'' the category of representations $\RModCop$.

\section{Foundational material}
The reader who is familiar with the language of topos may safely skip this section which is based on \cite{G} and mostly \cite{ML} .
Recall that given a small category $\cC$ on which there is a basis $K$ for a Grothendieck topology, an object $E$ of $\SetCop$ is a sheaf if and only if for any $\{f_i: C_i \rarr C \quad | \quad i \in I \}$ in $K(C)$, any matching family has a unique amalgamation. This means given:
\beq
\begin{CD}
C_i \times_C C_j @>\pi_{ij}^{(2)}>> C_j \\
@V\pi_{ij}^{(1)}VV @VVf_jV \\
C_i @>>f_i> C
\end{CD} \label{CDone}
\eeq
and assuming the existence of a matching family $\{ x_i \}_{i \in I}$, $E$ maps \eqref{CDone} into:
\beq
\setlength{\unitlength}{0.5cm}
\begin{picture}(16,9)(0,0)
\thicklines
\put(2.5,1){$E(C_i) \ni x_i$}
\put(12,1){$x \quad !$}
\put(12,7){$x_j \in E(C_j)$}
\put(0,7){$E(C_i \times_C C_j) \ni \alpha$}
\put(3,4){$E(\pi_{ij}^{(1)})$}
\put(8,8){$E(\pi_{ij}^{(2)})$}
\put(6,2){\vector(0,1){4}}
\multiput(7,1)(0.5,0){8}{\line(1,0){0.3}}
\put(7.5,1){\vector(-1,0){0.5}}
\multiput(12.2,2)(0,0.5){8}{\line(0,1){0.3}}
\put(12.2,5.5){\vector(0,1){0.5}}
\put(11,7){\vector(-1,0){4}}
\end{picture}
\nonumber
\eeq
In words, if $E(\pi_{ij}^{(1)})x_i = \alpha = E(\pi_{ij}^{(2)})x_j$, which is written $x_i \cdot \pi_{ij}^{(1)} = x_j \cdot \pi_{ij}^{(2)}$ in \cite{ML}, then $\exists \, ! \, x \in E(C)$ such that $x \cdot f_i = x_i$ and $x \cdot f_j = x_j$. Thus $E$ maps objects to sets, and if elements $x_i$ and $x_j$ do match in $E(C_i \times_C C_j)$, then there is a unique $x \in E(C)$ from which they are originating, thereby extending to the level of sets the commutative diagram provided by the pullback at the level of objects in $\cC$. For $\cC$ a small category, $J$ a Grothendieck topology on $\cC$, sheaves on $(\cC,J)$ along with natural transformations between them form a category of sheaves $\ShCJ$. A Grothendieck topos is a category $\cE$ equivalent to $\ShCJ$ for some site $(\cC,J)$. For the definitions of kernel pairs, exact forks and equivalence relations, the reader is referred to \cite{ML}.\\

\section{Giraud's Theorem}
We first state Giraud's Theorem and then give a proof different in spirit from the original one (\cite{ML}).
\begin{Giraud} \label{Gir}
Let $\cE$ be a category with small hom-sets and all finite limits. Then it is a Grothendieck topos if and only if the following properties hold:
(i) $\cE$ has all small coproducts and in addition those are disjoint, and stable under pullback.\\
(ii) Every epimorphism in $\cE$ is a coequalizer.\\
(iii) Every equivalence relation $R \rrarr E$ in $\cE$ is a kernel pair and admits a quotient.\\
(iv) Every exact fork $R \rrarr E \rarr Q$ is stably exact.\\
(v) There is a small set of objects of $\cE$ that generates $\cE$.
\end{Giraud}

That a Grothendieck topos satisfies those conditions is not difficult to show and this is done in \cite{ML} in particular so we will not repeat it here. What is more interesting is the converse statement. By (v) there is a small set of objects of $\cE$ that generates $\cE$. Let $\cC \subseteq \cE$ be the small, full subcategory of $\cE$ whose objects are those generators. Then for any $E$ object of $\cE$, the set of all morphisms $C_i \rarr E$, $C_i \in \Ob(\cC)$ forms an epimorphic family. By assumption $\cE$ has all finite limits, in particular it has pullbacks. We define a basis $K$ for a Grothendieck topology $J$ on $\cC$: $\forall C \in \Ob(\cC)$, $K(C)$ is an epimorphic family of maps into $C$, as suggested by having $\cC$ generate $\cE$.
\begin{Lemma}
$K$ is a basis for a Grothendieck topology on $\cC$.
\end{Lemma}
\begin{proof}
If $f: C' \rarr C$ is an isomorphism, then $\{C' \rarr C\}$ is an epimorphic family, so it belongs to $K(C)$. Consider $\{f_i: C_i \rarr C \quad | \quad i \in I \} \in K(C)$, and any morphism $g: D \rarr C$. Then considering the pullback:
\beq
\begin{CD}
(\coprod_{i \in I} C_i) \times_C D @>\tilde{f}>> D \\
@VVV @VVgV \\
\coprod_{i \in I} C_i @>>f> C
\end{CD}
\eeq
if $\{f_i \} \in K(C)$, then $f: \coprod_{i \in I} C_i \rarr C$ is an epi, and by (ii) it is a coequalizer of a pair of arrows, so it is a coequalizer of its kernel pair, hence we have an exact fork:
\beq
R  \rrarr \coprod_{i \in I} C_i \xrarr{f} C
\nonumber
\eeq
stable by pullback by (iv) so:
\beq
R \times_C D \rrarr (\coprod_{i \in I} C_i) \times_C D \xrarr{f \times 1} C \times_C D = D
\nonumber
\eeq
is exact, which implies that $\tilde{f}=f \times 1$ is a coequalizer, so it is an epi. At this point it is worth mentioning that $\cE$ has all coequalizers as proved in Appendix.2.1 of \cite{ML}, whose proof uses (iii), something we do not use otherwise. Now by (i) coproducts are stable by pullback so:
\beq
(\coprod_{i \in I} C_i) \times_C D = \coprod_{i \in I} C_i \times_C D
\nonumber
\eeq
thus $\coprod_{i \in I}(C_i \times_C D) \rarr D$ is an epi so $\{ C_i \times_C D \rarr D \quad | \quad i \in I \}$ is an epimorphic family, so it is in $K(D)$. Finally if $\{f_i : C_i \rarr C \quad | \quad i \in I \} \in K(C)$ and for all $i$ we have $\{ g_{ij}:D_{ij} \rarr C_i \quad | \quad j \in J_i \} \in K(C_i)$, then $f: \coprod_{i \in I} C_i \xrarr{f_i} C$ is an epi, $g_i: \coprod_{j \in J_i}D_{ij} \xrarr{g_{ij}} C_i$ is an epi for all $i \in I$, so in particular $g: \coprod_{i \in I}(\coprod_{j \in J_i}D_{ij}) \xrarr{g_i} \coprod_{i \in I}C_i$ is an epi, from which it follows by composition that
\beq
f \circ g: \coprod_{\substack{i \in I \\ j \in J_i}} D_{ij} \rarr C \nonumber
\eeq
is an epi as well, so that the family $\{D_{ij} \rarr C \}$ is in $K(C)$, which completes the proof of the lemma.
\end{proof}
Now we let $J = (K)$ be the Grothendieck topology on $\cC$ generated by $K$. In the classical proof one uses a Hom-tensor adjunction to prove that $\cE$ is a Grothendieck topos. At some point one proves that $\text{Hom}_{\cE}(-, E): \cC^{op} \rarr \text{Sets}$ is a sheaf for $J$. What we do instead is prove that $E$ itself is a sheaf by focusing on its structure rather than invoking the Yoneda embedding.\\

Since $\cC$ generates $\cE$, then an object $E$ of $\cE$ is equivalently given by any of the sets:
\beq
\Ed(C_i) = \{g_i: C_i \rarr E \} \nonumber
\eeq
for $C_i \in \Ob(\cC)$, as well as how the different images in $E$ glue. If we denote by $+$ the operation of collecting morphisms together and by $\cup$ the gluing, then one can write $E = \Ed(C_i)\mathbf{+} \cup_i$ for any $C_i \in \Ob(\cC)$, since the gluing depends on $i$. Now $\cE$ is generated by $\cC$ is a statement about objects of $\cE$. $\cE$ being generated by $\cC$ means two morphisms $u,v:E \rarr E'$ are different implies there is an object $C_i$ of $\cC$ and there is a morphism $f:C_i \rarr E$ such that $uf \neq vf$. Thus to fully characterize $E$ as an object of $\cE$, one needs all such sets $\Ed(C_i)$, thus $E$ is equivalently given by:
\begin{align}
\Ed \mathbf{+}\cup: \cC^{op} & \rarr \text{Sets} \nonumber \\
C_i & \mapsto \Ed(C_i) + \cup_i
\end{align}
or:
\beq
\coprod_{C_i \in \Ob(\cC)} \Ed(C_i) + \cup_i
\eeq
Now we will consider functors $\cC^{op} \rarr \text{Sets}$ that can be represented as $\coprod_{C_i \in \Ob(\cC)} \Ed(C_i)$ without mention of the gluing. Since we have dropped any notion of gluing on each $\Ed(C_i)$, the question is what structure do we need to add to $\coprod_{C_i \in \Ob(\cC)} \Ed(C_i)$ to obtain an object equivalent to $E$. We will first consider $E$ and see we can extract from it $\coprod_{C_i \in \Ob(C_i)} \Ed(C_i)$ with some additional structure. Then starting from some such coproduct we will deduce what needs to be added to get back $E$. Note that if subparts of $E$ glue in $\cE$, since $\cC$ generates $\cE$, then we have morphisms $x: C_i \rarr E$ and $y: C_j \rarr E$ that correspond to such a gluing, and gluing such maps in particular yields a map $C_i \times C_j \rarr E$ with a restriction $C_i \times_{E} C_j \rarr E$, hence a matching family:
\beq
\setlength{\unitlength}{0.5cm}
\begin{picture}(10,7)(0,0)
\thicklines
\put(2,0){$\Ed(C_i)$}
\put(0,5){$\Ed(C_i \times_E C_j)$}
\put(9,5){$\Ed(C_j)$}
\put(3,1){\vector(0,1){3}}
\put(8.5,5){\vector(-1,0){3}}
\end{picture}
\eeq
Note that the existence of such a matching family just means we have two morphisms $C_i \rarr E$, $C_j \rarr E$ that can be glued since they yield a morphism $C_i \times_E C_j \rarr E$, but it does not provide such a gluing. The gluing itself is given by a map $E \xrarr{\cup} E$, hence an element of $\Ed(E)$. If both $C_i \rarr E$ and $C_j \rarr E$ factor through $\cup$, then the composition provides the gluing of such maps. This also provides maps:
\beq
\setlength{\unitlength}{0.5cm}
\begin{picture}(10,7)(0,0)
\thicklines
\put(2,2){$\Ed(C_i)$}
\put(8.5,6){$\Ed(C_j)$}
\put(8.5,2){$\Ed(E)$}
\put(8,2){\vector(-1,0){3}}
\put(9,3){\vector(0,1){2.5}}
\end{picture}
\eeq
hence an amalgamation. Thus $E$ can be represented as:
\beq
E \simeq \bigcup_{\substack{C_i \times_E C_j \\ \text{amalg}}} \coprod_{C_i \in \Ob(\cC)} \Ed(C_i)
\eeq
Now from $\coprod_{C_i \in \Ob(\cC)} \Ed(C_i)$, matching families over $E$ provide an association of morphisms $C_i \rarr E$, $C_j \rarr E$ that can be glued. The existence of an amalgamation if unique provides such a gluing:
\beq
\bigcup_{\substack{C_i \times_E C_j \\ !\text{amalg}}} \coprod_{C_i \in \Ob(\cC)} \Ed(C_i) \simeq E \nonumber
\eeq
making the sheaf $\Ed = \Hom(-,E)$ equivalent to $E$. This will be the object of Proposition \ref{sheaf} below. Further if $\Ed(C) = \{g:C \rarr E \}$ and $f: C \rarr C'$ then we have an induced map:
\beq
\Ed(C) \xleftarrow{ \cdot f} \Ed(C') \ni x'
\eeq
with:
\begin{align}
x' \cdot f &= \{g':C' \rarr E \} \cdot f \nonumber \\
&=\{g'f: C \rarr C' \rarr E \}
\end{align}

\begin{sheaf} \label{sheaf}
Objects of $\cE$ are sheaves for $J$ on $\cC$ given above.
\end{sheaf}
\begin{proof}
Given:
\beq
\begin{CD}
C_i \times_C C_j @>\pi_{ij}^{(2)}>> C_j \\
@V\pi_{ij}^{(1)}VV @VVf_jV \\
C_i @>>f_i> C
\end{CD}
\eeq
and assuming we have a matching family as in:
\beq
\setlength{\unitlength}{0.5cm}
\begin{picture}(16,9)(0,0)
\thicklines
\put(2.5,1){$\Ed(C_i) \ni x_i$}
\put(12,1){$?$}
\put(12,7){$x_j \in \Ed(C_j)$}
\put(0,7){$\Ed(C_i \times_C C_j) \ni \alpha$}
\put(3,4){$\Ed(\pi_{ij}^{(1)})$}
\put(8,8){$\Ed(\pi_{ij}^{(2)})$}
\put(6,2){\vector(0,1){4}}
\multiput(7,1)(0.5,0){8}{\line(1,0){0.3}}
\put(7.5,1){\vector(-1,0){0.5}}
\multiput(12.2,2)(0,0.5){8}{\line(0,1){0.3}}
\put(12.2,5.5){\vector(0,1){0.5}}
\put(11,7){\vector(-1,0){4}}
\end{picture}
\nonumber
\eeq
we prove there is a unique amalgamation. The existence of a matching family can be rephrased as:
\begin{align}
x_i & = g_i: C_i \rarr E \nonumber \\
x_j & = g_j: C_j \rarr E
\end{align}
and $x_i \cdot \pi_{ij}^{(1)} = x_j \cdot \pi_{ij}^{(2)}$ means
\beq
\setlength{\unitlength}{0.5cm}
\begin{picture}(16,4)(0,0)
\thicklines
\put(0,0){$g_i \cdot \pi_{ij}^{(1)}: \, C_i \times_C C_j \xrarr{\pi_{ij}^{(1)}} C_i \xrarr{g_i} E$}
\put(1,1){\line(0,1){1}}
\put(1.2,1){\line(0,1){1}}
\put(0,2.5){$g_j \cdot \pi_{ij}^{(2)}: C_i \times_C C_j  \xrarr{\pi_{ij}^{(2)}} C_j \xrarr{g_j} E$}
\end{picture}
\eeq
so both squares below commute, the smaller one because it is the pullback in $\cC$ we started from, the larger one if we have a matching family:
\beq
\setlength{\unitlength}{0.5cm}
\begin{picture}(20,14)(0,0)
\thicklines
\put(7,5){$C_i$}
\put(14.5,5){$C$}
\put(15,11){$C_j$}
\put(5.5,11){$C_i \times_C C_j$}
\put(17.5,1){$E$}
\put(8,5){\vector(1,0){6}}
\put(15,10){\vector(0,-1){4}}
\put(7,10){\vector(0,-1){4}}
\put(9,11){\vector(1,0){5}}
\put(5,8){$\pi_{ij}^{(1)}$}
\put(11,12){$\pi_{ij}^{(2)}$}
\put(12,2){$g_i$}
\put(17.3,7){$g_j$}
\put(8,4.5){\vector(3,-1){8.5}}
\put(16,10){\vector(1,-4){1.8}}
\end{picture}
\eeq
which implies there is a unique map $r: C_i \times_C C_j \rarr C_i \times_E C_j$ as in:
\beq
\setlength{\unitlength}{0.5cm}
\begin{picture}(20,17)(0,0)
\thicklines
\put(7,5){$C_i$}
\put(14.5,5){$C$}
\put(14.5,10.5){$C_j$}
\put(5.5,11){$C_i \times_C C_j$}
\put(17.5,1){$E$}
\put(8,5){\vector(1,0){6}}
\put(15,10){\vector(0,-1){4}}
\put(7,10){\vector(0,-1){4}}
\put(9,11){\vector(1,0){5}}
\put(5,8){$\pi_{ij}^{(1)}$}
\put(11,12){$\pi_{ij}^{(2)}$}
\put(12,2){$g_i$}
\put(17.3,7){$g_j$}
\put(8,4.5){\vector(3,-1){8.5}}
\put(16,10){\vector(1,-4){1.8}}
\put(1,16){$C_i \times_E C_j$}
\multiput(6,12)(-0.3,0.3){11}{\circle*{0.15}}
\put(3,15){\vector(-1,1){0.5}}
\put(5.5,13){$\exists \, ! \, r$}
\qbezier(2,15)(2,5)(6,5)
\put(6,5){\vector(1,0){0.5}}
\qbezier(5,16)(15,16)(15,12)
\put(15,12){\vector(0,-1){0.5}}
\put(0.5,10){$\pi_{ij}^{(1)}$}
\put(10,16){$\pi_{ij}^{(2)}$}
\end{picture}
\eeq
Since maps from the fibered products are projections, it follows that $r$ is the identity on $C_i$ and $C_j$, but maps one fibered product into the other, so is a map on the structure of such products other than their factors, hence it is a map from $C_i \rarr C \leftarrow C_j$ to $C_i \rarr E \leftarrow C_j$, i.e. it is a map $C \rarr E$. Finally because $r$ is unique so is this map. So $\exists \, ! \, g:C \rarr E$ which is an element $x \in \Ed(C)$ such that $g \cdot f_j = g_j$ and $g \cdot f_i = g_i$, i.e. $x \cdot f_j = x_j$ and $x \cdot f_i = x_i$. This for all morphisms $C_i \rarr E$, $C_j \rarr E$, so $\Ed$ is a sheaf, hence so is $E \simeq \Ed$ by the argument preceding the Proposition,  and this for all objects $E$ of $\cE$, so objects of $\cE$ are sheaves.
\end{proof}

Now if $E,F \in \Ob(\cE)$, $\alpha: E \rarr F$ a morphism, we show it is a natural transformation of sheaves. Let $C$ be an object of $\cC$. Then $\alpha$ induces a map $\Ed(C) \xrarr{\alpha_C} \Fd(C)$. Further if $f:C \rarr C'$ is a morphism in $\cC$ then we have maps $\Ed(C) \xleftarrow{\cdot f} \Ed(C')$ and $\Fd(C) \xleftarrow{\cdot f} \Fd(C')$. We show the following square is commutative:
\beq
\setlength{\unitlength}{0.5cm}
\begin{picture}(10,8)(0,0)
\thicklines
\put(2,2){$\Ed(C')$}
\put(2,6){$\Ed(C)$}
\put(8.5,6){$\Fd(C)$}
\put(3,3){\vector(0,1){2.5}}
\put(2,4){$\cdot f$}
\put(5,6){\vector(1,0){3}}
\put(6,7){$\alpha_C$}
\put(8.5,2){$\Fd(C')$}
\put(5,2){\vector(1,0){3}}
\put(6,1){$\alpha_{C'}$}
\put(9,3){\vector(0,1){2.5}}
\put(9.5,4){$\cdot f$}
\end{picture}
\eeq
Denote by $p:C' \rarr E$ an element of $\Ed(C')$. Consider the following path:
\beq
\setlength{\unitlength}{0.5cm}
\begin{picture}(10,7)(0,1)
\thicklines
\put(2,2){$\Ed(C')$}
\put(2,6){$\Ed(C)$}
\put(8.5,6){$\Fd(C)$}
\put(3,3){\vector(0,1){2.5}}
\put(2,4){$\cdot f$}
\put(5,6){\vector(1,0){3}}
\put(6,7){$\alpha_C$}
\end{picture}
\eeq
Under this path by precomposition with $f$, $p$ is mapped to $p \circ f = p \cdot f$ in $\Ed(C)$, which is then mapped under $\alpha$ to $\alpha \circ ( p \circ f) \in \Fd(C)$. Now consider the lower path:
\beq
\setlength{\unitlength}{0.5cm}
\begin{picture}(10,7)(0,0)
\thicklines
\put(2,2){$\Ed(C')$}
\put(8.5,6){$\Fd(C)$}
\put(8.5,2){$\Fd(C')$}
\put(5,2){\vector(1,0){3}}
\put(6,1){$\alpha_{C'}$}
\put(9,3){\vector(0,1){2.5}}
\put(9.5,4){$\cdot f$}
\end{picture}
\eeq
Under this path $p$ is mapped by composition with $\alpha$ to $\alpha \circ p \in \Fd(C')$, which is then mapped by precomposition with $f$ to $(\alpha \circ p) \circ f \in \Fd(C)$. Composition being associative, those elements coincide and the square commutes, and this for all $C$, $C'$, so $\alpha$ is a natural transformation. Further by commutativity of such squares it also follows easily that $\alpha$ maps matching families to matching families and unique amalgamations to unique amalgamations, so $\alpha$ is a morphism of sheaves, and thus $\cE$ is a Grothendieck topos, which completes the proof of the theorem.

\section{Categories of Representations and Giraud}
For $\cC$ a small category, $R$ a commutative ring with a unit, the functor category $\RModCop$ is called a category of representations (\cite{X}, \cite{W}). One may ask what are the conditions on a category to be equivalent to such a category of representations. Before doing so we have to introduce some terminology.
\begin{abOp}
We define an abelian operad to be a 2-sided $\Sigma$-equivariant topological operad. The multiplication is denoted + which if commutative makes this operad $\Sigma$-equivariant, which motivates the name abelian. The operad is topological insofar as it is depicted by pairs of pants and maps isomorphic tensor products of objects $a$ and $b$ to $a+b$. It is 2-sided in the sense that $+(a \otimes b) = a+b $ can be read from left to right, or from right to left, as in:
\beq
\setlength{\unitlength}{0.5cm}
\begin{picture}(7,7)(0,0)
\thinlines
\put(1,6){\oval(2,2)[b]}
\put(6,6){\oval(2,2)[b]}
\thicklines
\thicklines
\put(1,6){\oval(2,2)[t]}
\put(6,6){\oval(2,2)[t]}
\put(3.5,1){\oval(3,2)}
\put(0,4){\line(0,1){2}}
\qbezier(0,4)(2,3)(2,1)
\put(3.5,6){\oval(3,4)[b]}
\put(7,4){\line(0,1){2}}
\qbezier(7,4)(5,3)(5,1)
\put(1,6){$a$}
\put(6,6){$b$}
\put(2.5,1){$a +b$}
\end{picture}
\eeq
which corresponds to grouping $a$ and $b$ according to $+(a \otimes b) = a+b$, or it can be represented as:
\beq
\setlength{\unitlength}{0.5cm}
\begin{picture}(7,7)(0,0)
\thinlines
\put(3.5,6){\oval(3,2)[b]}
\thicklines
\thicklines
\put(3.5,6){\oval(3,2)[t]}
\put(1,1){\oval(2,2)}
\put(6,1){\oval(2,2)}
\put(3.5,1){\oval(3,4)[t]}
\put(0,1){\line(0,1){2}}
\qbezier(0,3)(2,4)(2,6)
\put(7,1){\line(0,1){2}}
\qbezier(7,3)(5,4)(5,6)
\put(1,1){$a$}
\put(6,1){$b$}
\put(2.5,6){$a +b$}
\end{picture}
\eeq
\end{abOp}

\begin{cohabop}
A coherent abelian operad $\cO$ on a symmetric monoidal category $\cC$ with duals is an abelian operad whose operation is compatible with the existence of duals in $\cC$ in the following sense: the unit object 1 of $\cC$ is a unit for the operation $\bot$ of $\cO$: $a \bot 1 = 1 \bot a = a$, and duals map under $\cO$ to inverses for $\bot$: $a \otimes a^* \mapsto a \bot a^{-1} = 1$, $a^* \otimes a \mapsto a^{-1} \bot a = 1$ for any $a$.
\end{cohabop}
Note that since in a monoidal category we have natural isomorphisms $\rho: - \otimes 1 \Rarr id$ and $\lambda: 1 \otimes - \Rarr id$, the unit object is necessarily a unit for any operation $\bot$  of an abelian operad $\cO$ on $\cC$.
\begin{Onm}
Let $\cC$ be a fixed symmetric monoidal category with duals, $\cO$ an abelian operad defined on $\cC$. Denote $\cO$ by $\Oenm$ if it has $n$ operations, $m$ of which are coherent with the structure of $\cC$. $\Oenm$ is said to be sub-coherent if $n>m$.
\end{Onm}
\begin{umbrella}
Let $\cC$, $\cD$ be two symmetric monoidal categories with duals with respective sub-coherent abelian operads $\Oenm$ and $\Pqr$ defined on them. Construct a monoidal category $\cU$ by taking as objects $\Ob(\cU) = \Ob(\cC) \cup \Ob(\cD)$, $\otimes_{\cU}\big|_{\cC} = \otimes_{\cC}$, $\otimes_{\cU}\big|_{\cD} = \otimes_{\cD}$, the associator $A$ on $\cU$, as well as the natural isomorphisms $R:- \otimes_{\cU} 1 \Rarr id$ and $L: 1 \otimes_{\cU} - \Rarr id$ restrict to those on $\cC$ and $\cD$. Additionally, with a view towards applications to left $R$-modules, we regard tensor products as ordered in the sense that for $r \in \cD$, $a \in \cC$, $r \otimes a \in \cU$, but $a \otimes r$ is not defined. On $\cU$ we define a sub-coherent abelian operad $\Skl$ that restricts to $\Oenm$ on $\cC$ and $\Pqr$ on $\cD$, and constitutes an injective enrichment of coherent operads on $\cC$ and $\cD$ in the following sense: if $a,b \in \cC$ (resp. $\cD$), $\bot$ a coherent operation on $\cC$ (resp. $\cD$), there is a corresponding operation $\bot_{\cU}$ on $\cD \otimes \cC$ that generalizes $\bot$ such that if $\lambda \in \cU$ (resp. $\cD$), $\diamond$ any non-coherent operation on $\cU$ (resp. $\cD$):
\beq
\setlength{\unitlength}{0.5cm}
\begin{picture}(18,13)(0,0)
\thicklines
\put(0,4){$a \bot b$}
\put(6,4){$\lambda \diamond (a \bot b )$}
\put(6,0){$\lambda \otimes (a \bot b) $}
\put(14,4){$ \lambda \diamond a \bot  \lambda \diamond b$}
\put(14,8){$\lambda \diamond a \otimes  \lambda  \diamond b$}
\put(6,8){$\lambda \diamond(a \otimes b)$}
\put(0,8){$a \otimes b$}
\put(6,11.5){$\lambda \otimes (a \otimes b) $}
\put(3,4){\vector(1,0){2.8}}
\put(3,3){\vector(3,-2){3}}
\put(8,1){\vector(0,1){2}}
\put(8,7){\vector(0,-1){2}}
\put(10,4){\line(1,0){3}}
\put(10,4.5){\line(1,0){3}}
\put(15,7){\vector(0,-1){2}}
\put(10,8){\vector(1,0){3}}
\put(8,11){\vector(0,-1){2}}
\put(3,8){\vector(1,0){2.8}}
\put(3,9){\vector(3,2){3}}
\put(1,7){\vector(0,-1){2}}
\put(1.5,6){$\cO_{\bot}$}
\put(6,6){$\lambda  \cO_{\bot} $}
\put(11,8.5){$\sim$}
\multiput(10,7)(0.5,-0.2){8}{\circle*{0.2}}
\put(13.7,5.5){\vector(2,-1){0.5}}
\put(12,6.5){$\cO_{\bot_{\cU}}$}
\put(15.6,6){$\cO_{\bot_{\cU}}$}
\end{picture}
\eeq
This same operation $\bot_{\cU}$ satisfies the diagram below with $a,b \in \cC$ (resp. $\cD$), $\lambda \in \cC$ (resp. $\cU$), $\bot$ a coherent operation on $\cC$ (resp. $\cU$), $\diamond$ any non-coherent operation on $\cC$ (resp. $\cU$):
\beq
\setlength{\unitlength}{0.5cm}
\begin{picture}(18,13)(0,0)
\thicklines
\put(0,4){$a \bot b$}
\put(6,4){$(a \bot b ) \diamond \lambda$}
\put(6,0){$(a \bot b) \otimes \lambda$}
\put(14,4){$a \diamond \lambda \bot b \diamond \lambda$}
\put(14,8){$a \diamond \lambda \otimes b \diamond \lambda$}
\put(6,8){$(a \otimes b) \diamond \lambda$}
\put(0,8){$a \otimes b$}
\put(6,11.5){$(a \otimes b) \otimes \lambda$}
\put(3,4){\vector(1,0){3}}
\put(3,3){\vector(3,-2){3}}
\put(8,1){\vector(0,1){2}}
\put(8,7){\vector(0,-1){2}}
\put(10,4){\line(1,0){3}}
\put(10,4.5){\line(1,0){3}}
\put(15,7){\vector(0,-1){2}}
\put(10,8){\vector(1,0){3}}
\put(8,11){\vector(0,-1){2}}
\put(3,8){\vector(1,0){3}}
\put(3,9){\vector(3,2){3}}
\put(1,7){\vector(0,-1){2}}
\put(1.5,6){$\cO_{\bot}$}
\put(6,6){$\cO_{\bot} \lambda$}
\put(11,8.5){$\sim$}
\multiput(10,7)(0.5,-0.2){8}{\circle*{0.2}}
\put(13.7,5.5){\vector(2,-1){0.5}}
\put(12,6.5){$\cO_{\bot_{\cU}}$}
\put(15.6,6){$\cO_{\bot_{\cU}}$}
\end{picture}
\eeq
Thus non-coherent operations take precedence over coherent ones. In addition for non-coherent operations $\bot$ on $\cC$ or $\cD$, for $a,b \in \cC$ (resp. $\cD$), $\lambda \in \cC$ (resp. $\cU$), $\diamond$ any non-coherent operation on $\cU$:
\beq
\setlength{\unitlength}{0.5cm}
\begin{picture}(22,12)(0,0)
\thicklines
\put(1,2){$a \bot b$}
\put(8,2){$(a \bot b) \diamond \lambda $}
\put(17,2){$a \bot ( b \diamond \lambda)$}
\put(17,7){$a \bot(b \otimes \lambda)$}
\put(8,7){$(a \otimes b) \diamond \lambda$}
\put(1,7){$a \otimes b$}
\put(8,11){$(a \otimes b) \otimes \lambda$}
\put(17,11){$a \otimes ( b \otimes \lambda)$}
\put(4,2){\vector(1,0){3}}
\put(13,2){\line(1,0){3}}
\put(13,2.5){\line(1,0){3}}
\put(2,6){\vector(0,-1){3}}
\put(19,6){\vector(0,-1){3}}
\put(19,10){\vector(0,-1){2}}
\put(10,6){\vector(0,-1){3}}
\put(10,10){\vector(0,-1){2}}
\put(13,11){\vector(1,0){3}}
\put(4,7){\vector(1,0){3}}
\put(4,8){\vector(3,2){3}}
\put(14,10){$A$}
\put(1,5){$\bot$}
\put(18,5){$\diamond$}
\put(17.5,9){$\bot$}
\put(8.5,9){$\diamond$}
\put(9,5){$\bot$}
\end{picture}
\eeq
Such a category $\cU$ we call an umbrella category and we denote by $\cU = \widehat{\cD \otimes \cC}$.
\end{umbrella}

\begin{asscat}
Let $\Cnm$ be the category of symmetric monoidal categories with duals on which there is a sub-coherent abelian operad $\Oenm$ being defined. An umbrella category $\cU = \widehat{\cR \otimes \cC}$, $\cR \in \cC^{(2|1)}$, $\cC \in \cC^{(1|1)}$ we call an assembly category.
\end{asscat}
\begin{Noteasscat}
We do not define an operad in the sense of May (\cite{M}) as an operad defined in a monoidal category $\cD$, but on $\cD$, in the sense that objects of $\cD$ can be grouped together via an external operad $\cO$ to be defined.
\end{Noteasscat}

In what follows we refer to assembly categories of the form $\widehat{\cR \otimes \cC}$ as $\cR$-colored assembly categories. Observe that the operadic structure of $\cR \in \cC^{(2|1)}$ is given by a ring $R$. Conversely one can put a structure of a strict symmetric monoidal category on $R$ to produce an element $\cR \in \cC^{(2|1)}$ that by abuse of notation we will still denote by $R$. In the same manner an abelian group $M$ can be made into an element of $\cC^{(1|1)}$ by putting a structure of a strict symmetric monoidal category on it. More generally, one can put a structure of a strict umbrella category on $R \otimes M$, and in particular a structure of a strict assembly category, giving rise to something we call a strict $R$-colored assembly category. In establishing a connection between categories and categories of representations, it will be useful to have the following lemma:

\begin{RnCat} \label{RnCat}
There is a forgetful functor from the category of $\cR$-colored assembly categories to $\RMod$ for some commutative ring $R$ with a unit that encodes the operadic structure of $\cR$. Conversely every $R$-module $M$ corresponds to a strict $R$-colored assembly category.
\end{RnCat}
\begin{proof}
Let $\cC_R = \widehat{\cR \otimes \cC}$ be a $\cR$-colored assembly category, and $R$ the commutative ring with a unit encoding the operadic stucture of $\cR$. Then a grouping of $a,b \in \Ob(\cC)$ is done as in:
\beq
\setlength{\unitlength}{0.5cm}
\begin{picture}(7,9)(0,0)
\thinlines
\put(1,6){\oval(2,2)[b]}
\put(6,6){\oval(2,2)[b]}
\thicklines
\thicklines
\put(1,6){\oval(2,2)[t]}
\put(6,6){\oval(2,2)[t]}
\put(3.5,1){\oval(3,2)}
\put(0,4){\line(0,1){2}}
\qbezier(0,4)(2,3)(2,1)
\put(3.5,6){\oval(3,4)[b]}
\put(7,4){\line(0,1){2}}
\qbezier(7,4)(5,3)(5,1)
\put(1,8){$a$}
\put(3,8){$\otimes$}
\put(6,8){$b$}
\put(2.5,1){$a +b$}
\end{picture}
\eeq
Since the operad $\cO$ on $\cC$ is $\Sigma$-equivariant and two-sided, such operation defines an additive structure on $\Ob(\cC) = M$. If we denote by 0 the unit object of $\cC$ then we have a natural isomorphism $\rho: - \otimes 0 \Rarr id_{\cC}$ at the level of $\cC$ which translates to $a + 0 = a$ on $M$ as observed before:
\beq
\setlength{\unitlength}{0.5cm}
\begin{picture}(12,9)(0,0)
\thinlines
\put(1,6){\oval(2,2)[b]}
\put(11,6){\oval(2,2)[b]}
\thicklines
\thicklines
\put(1,6){\oval(2,2)[t]}
\put(6,6){\oval(2,2)[t]}
\put(3.5,1){\oval(3,2)}
\put(0,4){\line(0,1){2}}
\qbezier(0,4)(2,3)(2,1)
\put(3.5,6){\oval(3,4)[b]}
\put(7,4){\line(0,1){2}}
\qbezier(7,4)(5,3)(5,1)
\put(1,8){$a$}
\put(3,8){$\otimes$}
\put(6,8){$0$}
\put(2.5,1){$a +0$}
\put(8,4){$=$}
\put(11,6){\oval(2,2)[t]}
\put(10,6){\line(0,-1){5}}
\put(12,6){\line(0,-1){5}}
\put(11,1){\oval(2,2)}
\put(11,1){$a$}
\put(11,8){$a$}
\end{picture}
\eeq
Likewise the existence of a natural isomorphism $\lambda: 0 \otimes - \Rarr id_{\cC}$ would yield $0 + a = a$. Thus 0 is the unit for + in $M = (\Ob(\cC), +)$. Note that $M$ is independent of $\rho$ and $\lambda$. Any element $a$ of $\cC$ has a right dual $a^* = -a$ relative to $\cO$ by coherence, that is an inverse for $a$ in $(M,+)$. The existence of duals is accompanied by the existence of units and counits $\eta: 0 \rarr -a \otimes a $ and $\epsilon: a \otimes -a \rarr 0$ (\cite{Y}) which do not come into the definition of $M$.
\beq
\setlength{\unitlength}{0.5cm}
\begin{picture}(7,9)(0,0)
\thinlines
\put(1,6){\oval(2,2)[b]}
\put(6,6){\oval(2,2)[b]}
\thicklines
\thicklines
\put(1,6){\oval(2,2)[t]}
\put(6,6){\oval(2,2)[t]}
\put(3.5,6){\oval(3,4)[b]}
\put(3.5,6){\oval(7,8)[b]}
\put(1,8){$a$}
\put(3,8){$\otimes$}
\put(6,8){$-a$}
\end{picture}
\eeq
For the associativity:
\beq
\setlength{\unitlength}{0.5cm}
\begin{picture}(20,12)(0,0)
\thinlines
\put(1,9){\oval(2,2)[b]}
\put(4,9){\oval(2,2)[b]}
\put(8,9){\oval(2,2)[b]}
\put(12,9){\oval(2,2)[b]}
\put(16,9){\oval(2,2)[b]}
\put(19,9){\oval(2,2)[b]}
\thicklines
\thicklines
\put(1,9){\oval(2,2)[t]}
\put(2.5,9){\oval(1,2)[b]}
\put(4,9){\oval(2,2)[t]}
\put(6,9){\oval(2,8)[b]}
\put(8,9){\oval(2,2)[t]}
\put(5,1){\oval(4,2)}
\put(0,9){\line(0,-1){3}}
\qbezier(0,6)(3,3)(3,1)
\put(9,9){\line(0,-1){4}}
\qbezier(9,5)(7,3)(7,1)
\put(0.5,11){$(a \quad \otimes \quad b)$}
\put(6,11){$\otimes$}
\put(8,11){$c$}
\put(3.1,0.7){$(a+b)+c$}
\put(8,1.2){\line(1,0){4}}
\put(8,0.8){\line(1,0){4}}
\put(9,11){\vector(1,0){2}}
\put(10,11.5){$\sim$}
\put(12,9){\oval(2,2)[t]}
\put(14,9){\oval(2,8)[b]}
\put(16,9){\oval(2,2)[t]}
\put(17.5,9){\oval(1,2)[b]}
\put(19,9){\oval(2,2)[t]}
\put(15,1){\oval(4,2)}
\put(11,9){\line(0,-1){4}}
\qbezier(11,5)(13,3)(13,1)
\put(20,9){\line(0,-1){3}}
\qbezier(20,6)(17,3)(17,1)
\put(12,11){$a$}
\put(14,11){$\otimes$}
\put(15.5,11){$(b \quad \otimes \quad c)$}
\put(13.1,0.7){$a+(b+c)$}
\end{picture}
\eeq
since $\cO$ maps isomorphic tensor products to the same element, hence we have associativity of + in $M$. If we denote by $\alpha$ the associativity isomorphism, then $M$ is independent of $\alpha$. For the commutativity:
\beq
\setlength{\unitlength}{0.5cm}
\begin{picture}(17,9)(0,0)
\thinlines
\put(1,6){\oval(2,2)[b]}
\put(6,6){\oval(2,2)[b]}
\put(11,6){\oval(2,2)[b]}
\put(16,6){\oval(2,2)[b]}
\thicklines
\thicklines
\put(1,6){\oval(2,2)[t]}
\put(6,6){\oval(2,2)[t]}
\put(3.5,1){\oval(3,2)}
\put(0,4){\line(0,1){2}}
\qbezier(0,4)(2,3)(2,1)
\put(3.5,6){\oval(3,4)[b]}
\put(7,4){\line(0,1){2}}
\qbezier(7,4)(5,3)(5,1)
\put(1,8){$a$}
\put(3,8){$\otimes$}
\put(6,8){$b$}
\put(2.5,1){$a +b$}
\put(11,6){\oval(2,2)[t]}
\put(16,6){\oval(2,2)[t]}
\put(13.5,1){\oval(3,2)}
\put(10,4){\line(0,1){2}}
\qbezier(10,4)(12,3)(12,1)
\put(13.5,6){\oval(3,4)[b]}
\put(17,4){\line(0,1){2}}
\qbezier(17,4)(15,3)(15,1)
\put(11,8){$b$}
\put(13,8){$\otimes$}
\put(16,8){$a$}
\put(12.5,1){$b+a$}
\put(6.5,1.2){\line(1,0){4}}
\put(6.5,0.8){\line(1,0){4}}
\put(8,8){\vector(1,0){2}}
\put(9,8.5){$\sim$}
\end{picture}
\eeq
since we have a braiding $\sigma$. $M$ is independent of $\sigma$ and whether $\cC$ is symmetric. The pentagon and triangle coherence conditions (\cite{Y}) are compatible with the $\mathbb{Z}$-module structure on $M$, and $M$ does not depend on such coherence conditions. For the right duality we have the unit and counit $\epsilon: a \otimes a^* \rarr 0$ and $\eta: 0 \rarr a^* \otimes a$ satisfy the following commutative diagrams with $a^* = -a$:
\beq
\setlength{\unitlength}{0.5cm}
\begin{picture}(17,5)(0,0)
\thicklines
\thicklines
\put(0,4){$a$}
\put(1,4){\vector(1,0){2}}
\put(2,4.5){$\rho^{-1}$}
\put(4,4){$a \otimes 0$}
\put(7,4){\vector(1,0){3}}
\put(8,4.5){$a \otimes \eta$}
\put(11,4){$a \otimes ( a^* \otimes a)$}
\put(0,3.5){\line(0,-1){2.5}}
\put(0.2,3.5){\line(0,-1){2.5}}
\put(13,3.5){\vector(0,-1){2.5}}
\put(14,2){$\alpha^{-1}$}
\put(0,0){$a$}
\put(3,0){\vector(-1,0){2}}
\put(2,0.5){$\lambda$}
\put(4,0){$0 \otimes a$}
\put(10,0){\vector(-1,0){3}}
\put(8,0.5){$\epsilon \otimes a$}
\put(11,0){$(a \otimes a^*) \otimes a$}
\end{picture}
\eeq
and:
\beq
\setlength{\unitlength}{0.5cm}
\begin{picture}(17,7)(0,-1)
\thicklines
\thicklines
\put(0,4){$a^*$}
\put(1,4){\vector(1,0){2}}
\put(2,4.5){$\lambda^{-1}$}
\put(4,4){$0 \otimes a^*$}
\put(7,4){\vector(1,0){3}}
\put(8,4.5){$\eta \otimes a^*$}
\put(11,4){$(a^* \otimes a) \otimes a^* $}
\put(0,3.5){\line(0,-1){2.5}}
\put(0.2,3.5){\line(0,-1){2.5}}
\put(13,3.5){\vector(0,-1){2.5}}
\put(14,2){$\alpha$}
\put(0,0){$a^*$}
\put(3,0){\vector(-1,0){2}}
\put(2,0.5){$\rho$}
\put(4,0){$a^* \otimes 0$}
\put(10,0){\vector(-1,0){3}}
\put(8,0.5){$a^* \otimes \epsilon$}
\put(11,0){$a^* \otimes (a \otimes a^*)$}
\end{picture}
\eeq
and those can be imposed independently of the existence of $\cO$. $M$ is independent of such coherence diagrams. For the braided structure we have a braiding $\sigma: \otimes \Lrarr \otimes (\text{twist})$ that satisfies the symmetry condition $\sigma_{b,a} \circ \sigma_{a, b} = id_{a \otimes b}$ with the 2 pentagon diagrams commuting (\cite{Y}), and $M$ is independent of those as well. For the $R$-module structure $\cC_R$ is a $\widehat{\cR \otimes \cC}$ umbrella category. The category $\cR \in \cC^{(2|1)}$ has one coherent operation + with unit $\eset$ and one non-coherent operation $\cdot$ with unit $\bland$. The addition on $\cC$ being coherent, $\cC_R$ being an umbrella category, for $a,b \in \Ob(\cC)$, $r \in \cR$, we have:
\beq
\setlength{\unitlength}{0.5cm}
\begin{picture}(18,13)(0,0)
\thicklines
\put(0,4){$a \bot b$}
\put(6,4){$r(a+b)$}
\put(6,0){$r \otimes (a + b)$}
\put(14,4){$ra + rb$}
\put(14,8){$ra \otimes rb$}
\put(6,8){$r(a \otimes b)$}
\put(0,8){$a \otimes b$}
\put(6,11.5){$r \otimes (a \otimes b) $}
\put(3,4){\vector(1,0){3}}
\put(3,3){\vector(3,-2){3}}
\put(8,1){\vector(0,1){2}}
\put(8,7){\vector(0,-1){2}}
\put(10,4){\line(1,0){3}}
\put(10,4.5){\line(1,0){3}}
\put(15,7){\vector(0,-1){2}}
\put(10,8){\vector(1,0){3}}
\put(8,11){\vector(0,-1){2}}
\put(3,8){\vector(1,0){3}}
\put(3,9){\vector(3,2){3}}
\put(1,7){\vector(0,-1){2}}
\multiput(10,7)(0.5,-0.2){8}{\circle*{0.2}}
\put(13.7,5.5){\vector(2,-1){0.5}}
\end{picture}
\eeq
hence $r(a+b) = ra + rb$. The addition being coherent in $\cR$ as well, for $r, r' \in \cR$, $a \in \cC$, we have:
\beq
\setlength{\unitlength}{0.5cm}
\begin{picture}(18,13)(0,0)
\thicklines
\put(0,4){$r+r'$}
\put(6,4){$(r+r')a$}
\put(6,0){$(r+r')\otimes a$}
\put(14,4){$ra + r'a$}
\put(14,8){$ra \otimes r'a$}
\put(6,8){$(r \otimes r')a$}
\put(0,8){$r \otimes r'$}
\put(6,11.5){$(r \otimes r') \otimes a $}
\put(3,4){\vector(1,0){3}}
\put(3,3){\vector(3,-2){3}}
\put(8,1){\vector(0,1){2}}
\put(8,7){\vector(0,-1){2}}
\put(10,4){\line(1,0){3}}
\put(10,4.5){\line(1,0){3}}
\put(15,7){\vector(0,-1){2}}
\put(10,8){\vector(1,0){3}}
\put(8,11){\vector(0,-1){2}}
\put(3,8){\vector(1,0){3}}
\put(3,9){\vector(3,2){3}}
\put(1,7){\vector(0,-1){2}}
\multiput(10,7)(0.5,-0.2){8}{\circle*{0.2}}
\put(13.7,5.5){\vector(2,-1){0.5}}
\end{picture}
\eeq
hence $(r+r')a = ra + r'a$ in $M$. For the $\cdot$, non-coherent in $\cR$, since $\cC_R$ is an umbrella category, for $a \in \cC$, we have:
\beq
\setlength{\unitlength}{0.5cm}
\begin{picture}(22,12)(0,0)
\thicklines
\put(1,2){$rr'$}
\put(9,2){$rr'a $}
\put(18,2){$r (r'a)$}
\put(17,7){$r \otimes(r'a)$}
\put(8,7){$(r \otimes r') a$}
\put(1,7){$r \otimes r'$}
\put(8,11){$(r \otimes r') \otimes a$}
\put(17,11){$r \otimes ( r' \otimes a)$}
\put(4,2){\vector(1,0){3}}
\put(13,2){\line(1,0){3}}
\put(13,2.5){\line(1,0){3}}
\put(2,6){\vector(0,-1){3}}
\put(19,6){\vector(0,-1){3}}
\put(19,10){\vector(0,-1){2}}
\put(10,6){\vector(0,-1){3}}
\put(10,10){\vector(0,-1){2}}
\put(13,11){\vector(1,0){3}}
\put(4,7){\vector(1,0){3}}
\put(4,8){\vector(3,2){3}}
\end{picture}
\eeq
Hence $(rr')a = r(r'a)$. Commutativity:
\beq
\setlength{\unitlength}{0.5cm}
\begin{picture}(22,12)(0,0)
\thicklines
\put(1,2){$rr'$}
\put(9,2){$rr'a $}
\put(18,2){$r'ra$}
\put(17,7){$(r' \otimes r)a)$}
\put(8,7){$(r \otimes r') a$}
\put(1,7){$r \otimes r'$}
\put(8,11){$(r \otimes r') \otimes a$}
\put(17,11){$(r' \otimes  r) \otimes a)$}
\put(4,2){\vector(1,0){3}}
\put(13,2){\line(1,0){3}}
\put(13,2.5){\line(1,0){3}}
\put(2,6){\vector(0,-1){3}}
\put(19,6){\vector(0,-1){3}}
\put(19,10){\vector(0,-1){2}}
\put(10,6){\vector(0,-1){3}}
\put(10,10){\vector(0,-1){2}}
\put(13,11){\vector(1,0){3}}
\put(4,7){\vector(1,0){3}}
\put(4,8){\vector(3,2){3}}
\end{picture}
\eeq
Hence $(rr')a = (r'r)a$. For the units:
\beq
\begin{CD}
(r \otimes \bland) \otimes a @>A>> r \otimes ( \bland \otimes a) \\
@VVV @VVV \\
r \otimes a @. r \otimes (id_a a) = r \otimes a\\
@VVV @VVV \\
ra @= ra
\end{CD}
\eeq
and for $r,r',r'' \in \cR$:
\beq
\setlength{\unitlength}{0.5cm}
\begin{picture}(18,13)(0,0)
\thicklines
\put(0,4){$r+r'$}
\put(6,4){$(r+r')r''$}
\put(6,0){$(r+r')\otimes r''$}
\put(14,4){$rr'' + r'r''$}
\put(14,8){$rr'' \otimes r'r''$}
\put(6,8){$(r \otimes r')r''$}
\put(0,8){$r \otimes r'$}
\put(6,11.5){$(r \otimes r') \otimes r'' $}
\put(3,4){\vector(1,0){3}}
\put(3,3){\vector(3,-2){3}}
\put(8,1){\vector(0,1){2}}
\put(8,7){\vector(0,-1){2}}
\put(10,4){\line(1,0){3}}
\put(10,4.5){\line(1,0){3}}
\put(15,7){\vector(0,-1){2}}
\put(10,8){\vector(1,0){3}}
\put(8,11){\vector(0,-1){2}}
\put(3,8){\vector(1,0){3}}
\put(3,9){\vector(3,2){3}}
\put(1,7){\vector(0,-1){2}}
\multiput(10,7)(0.5,-0.2){8}{\circle*{0.2}}
\put(13.7,5.5){\vector(2,-1){0.5}}
\end{picture}
\eeq
so $(r+r')r'' = rr'' + r'r''$. Also $\bland r = (\eset + \bland)r = \eset r + \bland r \Rarr \eset r = \eset$. We also have $\eset a = 0$ since $\bland a = (\eset + \bland)a = \eset a + \bland a$. Finally $M$ is independent of coherence conditions on $\cR$, $\cC$, $\cR \otimes \cC$, as well as the isomorphisms $\alpha$, $\rho$, $\lambda$, $\eta$, $\epsilon$, $\alpha_R$, $\rho_R$, $\lambda_R$, $\epsilon_R$, $\eta_R$, $A$, $R$ and $L$, or in other terms there is a forgetful functor from the category of $\cR$-colored assembly categories to $\RMod$ for $R$ encoding the operadic structure of $\cR$. Conversely, if all the above isomorphisms are the identity and all coherence conditions involve identity maps, we have a strict assembly category that can be obtained from a $R$-module alone. That is a $R$-module gives rise to a strict $R$-colored assembly category.
\end{proof}

Before stating the main theorem of this section, we define a particular module structure. We regard elements $r$ of a ring $R$ as morphisms with domains dom($r$) and we define $\text{Dom}(R) = \cup_{r \in R} \text{dom}(r)$. We consider $R$-modules $M$ such that $\underline{M} \subset \text{Dom}(R)$. Among objects of $\RNatModCop$, we consider those functors $F: \cC^{op} \rarr \RNatMod$ with $F(C) = \{ D \rarr E_C \}$, $E_C \equiv E$ for all $C \in \Ob(\cC)$. The subcategory of all such functors in $\RNatModCop$ is denoted $\sRNatModCop$. On this subcategory we define an equivalence relation by $F \sim F'$ if and only if $F(C)$ and $F'(C)$ are sets of morphisms into a same object $E$ of $\cE$ for any $C \in \Ob(\cC)$. We denote by $\RNatMod$ the set of such $R$-modules and we refer to such modules as $R$-natural modules.\\

We now state our main theorem:
\begin{equ}
Let $\cE$ be a category with small hom-sets. Then $\cE$ is equivalent to a moduli category of $R$-natural representations of a small category $\cC^{op}$ for some commutative ring $R$, $\sRNatModCop / \sim$, if and only if $\cE$ has a small set of generators.
\end{equ}
We will first prove something more general, that a category of representations has a small set of generators, and to achieve this we will show that any element of $\RModCop$ is a colimit of ``representables". We will first need to prove the Yoneda lemma for functors valued in the category of $R$-modules. Let $\cC$ be a small category. Let $R$ be a commutative ring with a unit. Let $h_C = \Hom(-,C)$ be the usual Yoneda embedding. Morphisms into $C$ are viewed as objects, that are equivalently given by a selection process, geared towards picking from an array of mathematical objects that particular morphism that is needed. One can make a selection, or cancel such a selection, giving rise to an abelian group structure on the set of selections. Denote by:
\beq
\cO h_C = \Mod(-, C)
\eeq
defined by $\cO h_C(X) = \Mod(X,C)$, $X \in \Ob(\cC)$ the abelian group of selections of morphisms into $C$. Making this group into an obvious, formal $R$-module defines a contravariant functor:
\begin{align}
\ROhC = \RMod(-,C): \cC^{op} &\rarr \RMod \nonumber \\
X & \mapsto \RMod(X,C)
\end{align}
where $\RMod(X,C)$ is the group $\Mod(X,C)$ endowed with the structure of an obvious $R$-module. This leads us to make a very important observation: in this work an element $F \in \RModCop$ is seen as being valued in $\RMod$ insofar as it reflects the $R$-module structure on the set of selections of objects and morphisms in $\cC$. We will be considering $R$-linear functors $\cC^{op} \rarr \RMod$. Thus in what follows when we discuss $\RModCop$ it is understood that we are considering functors $\cC^{op} \rarr \RMod$ that are $R$-linear as viewed as functors from the category of selection objects of $\cC$ endowed with its $R$-module structure, into $\RMod$. We now state and prove Yoneda's lemma for categories of representations:
\begin{Yoneda}
Let $F: \cC^{op} \rarr \RMod$ be a $R$-linear contravariant functor. Then $\forall C \in \Ob(\cC)$ we have a natural isomorphism:
\beq
\text{Nat}(R-\cO h_C, F) \simeq F(C)
\eeq
\end{Yoneda}
\begin{proof}
Let $C$ and $X$ be objects of $\cC$, $f:X \rarr C$ a morphism in $\cC$, $\gamma: \RMod(-,C) \Lrarr F$ a natural transformation from $\ROhC$ to $F$. We have the following commutative diagram:
\beq
\setlength{\unitlength}{0.5cm}
\begin{picture}(18,12)(0,0)
\thicklines
\put(5,0.7){$\vdash$}
\put(5,1){\vector(1,0){6}}
\put(4,8){\vector(0,-1){4}}
\put(12,8){\vector(0,-1){4}}
\put(7,9){\vector(1,0){3}}
\put(5.5,3){\vector(1,0){6}}
\put(6,10.7){$\vdash$}
\put(6,11){\vector(1,0){5}}
\qbezier(13,11)(20,11)(20,6)
\qbezier(20,6)(20,1)(18,1)
\put(18,1){\vector(-1,0){0.5}}
\qbezier(3,11)(0,11)(0,6)
\qbezier(0,6)(0,1)(3,1)
\put(3,1){\vector(1,0){0,5}}
\put(4,1){$\alpha$}
\put(11.5,1){$F(f)\alpha = \gamma_X(f)$}
\put(4,2){$\cup$}
\put(4.25,2){\line(0,1){0.5}}
\put(3,3){$F(C)$}
\put(7,3.5){$F(f)$}
\put(12,2){$\cup$}
\put(12.25,2){\line(0,1){0.5}}
\put(12,3){$F(X)$}
\put(3,6){$\gamma_C$}
\put(13,6){$\gamma_X$}
\put(1,9){$\RMod(C,C)$}
\put(6.5,8){$\ROhC(f)$}
\put(10.5,9){$\RMod(X,C)$}
\put(4,10){$\cap$}
\put(4.25,10.5){\line(0,-1){0.5}}
\put(12,10){$\cap$}
\put(12.25,10.5){\line(0,-1){0.5}}
\put(4,11){$id_C$}
\put(12,11){$f$}
\end{picture}
\eeq
We have a map
\begin{align}
\Nat(\ROhC, F) & \rarr F(C) \nonumber \\
\gamma & \mapsto \gamma_C(\text{id}_C) = \alpha
\end{align}
as shown in the above commutative diagram on the left. Let $M$ be a $R$-module, $\underline{M}$ its underlying set, $\str_M$ the $R$-module structure on $M$. Regarding elements of $\underline{M}$ as equivalently being given by their respective selection objects, one can add to $\underline{M}$ an additional selection object corresponding to selecting $\str_M$ on $\underline{M}$ to obtain $M$, resulting in another set denoted $\underline{M}^+$. The map above can then be represented as:
\begin{align}
\Nat(\ROhC, F) &\rarr \underline{F(C)}^+ \simeq F(C) \nonumber \\
\gamma &\mapsto \left\{
                  \begin{array}{ll}
                    \gamma_C(\text{id}_C) = \alpha \\
                    \gamma_C(\str_{\RMod(C,C)}) = \str_{F(C)}
                  \end{array}
                \right.
\end{align}
since $\gamma_C$ is a $R$-module homomorphism. Conversely, any $\alpha \in F(C)$ gives rise to some $\gamma \in \Nat(\ROhC, F)$:
\begin{align}
F(C) &\rarr \Nat(\ROhC,F) \nonumber \\
\alpha &\mapsto \gamma: \gamma_X(f) = F(f) \alpha
\end{align}
as in the commutative diagram, and in particular:
\begin{align}
\str_{F(C)} \rarr \gamma: \gamma_X(f) &= F(f) \str_{F(C)} \nonumber \\
&= \str_{F(X)} \text{ for all } f \in \RMod(X,C)
\end{align}
so $\gamma_X = \str_{F(X)}$, hence $\gamma$ produced by $\str_{F(C)}$ gives $F$ its structure of $\RMod$-valued functor in $\cC$. Moreover for $r,r' \in R$, $f':X \rarr C$:
\begin{align}
\gamma_X(rf + r'f') &= F(rf + r'f') \nonumber \\
&=\big( rF(f) + r'F(f') \big) \alpha \text{ by $R$-linearity} \nonumber \\
&=rF(f) \alpha + r' F(f') \alpha \nonumber \\
&=r \gamma_X(f) + r' \gamma_X(f')
\end{align}
hence $\gamma_X$ is a $R$-module homomorphism. Thus we have a map:
\begin{align}
\underline{F(C)}^+ \simeq F(C) & \rarr \Nat(\ROhC,F) \nonumber \\
\alpha & \mapsto \gamma \text{ such that } \gamma_X(f) = F(f) \alpha \nonumber \\
\str_{F(C)} & \mapsto  \str_F
\end{align}
These two maps constructed above being inverse to each other by commutativity, we have $\Nat(\ROhC, F) \simeq F(C)$.
\end{proof}

\begin{embed}
$\ROh: \cC \rarr \RModCop$ is fully faithful.
\end{embed}
\begin{proof}
It suffices to take $F = \ROh_Y$, $R$-linear, in the preceding lemma. Then:
\beq
\text{Nat}(\ROh_X, \ROh_Y) \simeq \ROh_Y(X) = \RMod(X,Y)
\eeq
\end{proof}
Now recall that in proving that every presheaf is a colimit of representables we typically make use of a category of elements. We use a generalization of such a category in our setting. For $F$ an object of $\RModCop$, we denote by $\int F$ and we call the category of elements of $F$ the category whose objects are pairs $(C,p)$ where $C$ is an object of $\cC$ and $p \in F(C)$ where $F(C)$ is an $R$-module. Thus one can also consider $p = \str_{F(C)} \in \underline{F(C)}^+ \simeq F(C)$. Morphisms $(C',p') \rarr (C, p)$ are morphisms $u:C' \rarr C$ along with the requirement on the elements of $F(C)$ and $F(C')$ that $F(u)p = p'$, or equivalently written, that $pu = p'$. We have a forgetful functor $\int F \xrarr{\pi_F} \cC$ that sends $(C,p)$ to $C$. We now prove the representation theoretic analog of a same statement for presheaf as proven in \cite{ML}. Note that we follow their proof, which almost carries over to the case of representation functors. We first prove:
\begin{rep} \label{rep}
For $\cC$ a small category, $R$ a commutative ring, $\ROh: \cC \rarr \RModCop = \cE$, then the functor:
\begin{align}
r: \cE &\rarr \cE \nonumber \\
E &\mapsto r(E): C \mapsto \Hom_{\cE}(\ROhC, E) = \text{Nat}(\ROhC,E)
\end{align}
has a left adjoint $l: \cE \rarr \cE$ defined for all $F \in \cE$ as the colimit:
\beq
l(F) = \text{colimit}(\int_{\cC}F \xrarr{\pi_F} \cC \xrarr{\ROh} \cE)
\eeq
\end{rep}
\begin{proof}
First observe that $R$ commutative implies that if $M$ and $N$ are two $R$-modules then $\Hom_R(M,N)$ can be endowed with a $R$-module structure. For $F, G \in \cE$, $\Nat(F,G)$ being defined pointwise, we can extend the $R$-module structure from $\Hom_R(F(C),G(C))$ for any $C$ to $\Nat(F,G)$. To wit $f: F(C) \rarr G(C)$, for $r \in R$, $\alpha \in F(C)$, then $(rf)\alpha = rf(\alpha)$, in particular for $\gamma : \ROhC \Rarr E$, $(r \gamma_C)(\alpha) = r \gamma_C(\alpha)$ for any $C$, so $r \gamma$ is well-defined. Hence $r(E) \in \cE = \RModCop$ is well-defined. Fix $E$ an object of $\cE$. Let $F$ be an object of $\cE$. Then a natural transformation $\phi: F \Rarr r(E)$ consists of a family $\{\phi_C \}$ of morphisms of $R$-modules such that for $u:C' \rarr C$ in $\cC$ we have a commutative diagram in $\RMod$:
\beq
\begin{CD}
F(C) @>\phi_C>> \Hom_{\cE}(\ROhC, E) \\
@VF(u)VV @VV\ROh(u)^*V \\
F(C') @>>\phi_{C'}> \Hom_{\cE}(\ROh_{C'}, E)
\end{CD} \label{phic}
\eeq
Note that for $p \in F(C)$, $\phi_C(p):\ROhC \rarr E$ and if we consider the category $\int F$ of elements of $F$, $p' = pu = F(u)p$, then commutativity of the above diagram can be recast in the following form:
\beq
\phi_{C'} (p') = \phi_{C'} \circ F(u)p = \ROh(u)^* \phi_C (p)
\eeq
i.e. $\ROh(u): \ROh_{C'} \rarr \ROhC$ induces $\{\phi_{C'}(p'): \ROh_{C'} \rarr E \}= \phi_{C}(p)\ROhC(u) $, or:
\beq
\setlength{\unitlength}{0.5cm}
\begin{picture}(13,8)(0,0)
\thicklines
\put(4,0){$\ROh_{C'}$}
\put(9,1){$\phi_{C'}(p')$}
\put(12,4){$E$}
\put(9,7){$\phi_C(p)$}
\put(0,4){$\ROh(u)$}
\put(4,7){$\ROhC$}
\put(5,1.5){\vector(0,1){4.5}}
\put(7.5,6.5){\vector(3,-1){3.5}}
\put(7.5,1.5){\vector(3,2){3.5}}
\end{picture}
\eeq
Rewrite this as a cocone:
\beq
\setlength{\unitlength}{0.5cm}
\begin{picture}(13,8)(0,0)
\thicklines
\put(1,0){$\ROh \pi_F(C',p')$}
\put(9,1){$\phi_{C'}(p')$}
\put(12,4){$E$}
\put(9,7){$\phi_C(p)$}
\put(3,4){$u_*$}
\put(1,7){$\ROh \pi_F(C,p)$}
\put(5,1.5){\vector(0,1){4.5}}
\put(7.5,6.5){\vector(3,-1){3.5}}
\put(7.5,1.5){\vector(3,2){3.5}}
\end{picture}
\eeq
Repeating the same argument from \eqref{phic} for $p= \str_{F(C)} \in \underline{F(C)}^+ \simeq F(C)$, $\phi_C(p) = \str_{\Nat(\ROhC,E)} \in \underline{\Hom_{\cE}(\ROhC, E)}^+ \simeq \Hom_{\cE}(\ROhC,E)$, which makes every element $\ROhC \rarr E$ of the set $\underline{\Hom_{\cE}(\ROhC,E)}$ an element of the $R$-module $\Hom_{\cE}(\ROhC,E)$, something we can represent as a decoration $\ROhC \xrarr{\phi_C(\str_{F(C)})} E$. In the same manner, $\ROh(u)^*$ maps the $R$-module structure of $\Hom_{\cE}(\ROhC, E)$ onto that of $\Hom_{\cE}(\ROh_{C'},E)$, i.e
\beq
\phi_{C'}(\str_{F(C')}) = \ROh(u)^* \phi_C(\str_{F(C)})
\eeq
i.e.
\beq
\setlength{\unitlength}{0.5cm}
\begin{picture}(7,5)(0,0)
\thicklines
\put(1,5){$\phi_C(\str_{F(C)})$}
\put(4,3){$\ROh(u)^*$}
\put(1,1){$\phi_{C'}(\str_{F(C')})$}
\put(3,4){\vector(0,-1){2}}
\end{picture}
\eeq
or equivalently represented:
\beq
\setlength{\unitlength}{0.5cm}
\begin{picture}(8,9)(0,0)
\thicklines
\put(0,2){$\ROh_{C'}$}
\put(3,1){$\phi_{C'}(\str_{F(C')})$}
\put(7,2){$E$}
\put(5,5){$\ROh(u)^*$}
\put(0,8){$\ROhC$}
\put(3,9){$\phi_C(\str_{F(C)})$}
\put(7,8){$E$}
\put(3,2){\vector(1,0){3}}
\put(4,7){\vector(0,-1){4}}
\put(3,8){\vector(1,0){3}}
\end{picture}
\eeq
$\ROh(u)^*$ being given by precomposition with $\ROh(u)$:
\beq
\setlength{\unitlength}{0.5cm}
\begin{picture}(13,8)(0,0)
\thicklines
\put(4,0){$\ROh_{C'}$}
\put(9,1){$\phi_{C'}(\str_{F(C')})$}
\put(12,4){$E$}
\put(9,7){$\phi_C(\str_{F(C)})$}
\put(0,4){$\ROh(u)$}
\put(4,7){$\ROhC$}
\put(5,1.5){\vector(0,1){4.5}}
\put(7.5,6.5){\vector(3,-1){3.5}}
\put(7.5,1.5){\vector(3,2){3.5}}
\end{picture}
\eeq
in particular:
\beq
\setlength{\unitlength}{0.5cm}
\begin{picture}(13,8)(0,0)
\thicklines
\put(1,0){$\ROh \pi_F(C',\str_{F(C')})$}
\put(9,1){$\phi_{C'}(\str_{F(C')})$}
\put(12,4){$E$}
\put(9,7){$\phi_C(\str_{F(C)})$}
\put(3,4){$u_*$}
\put(1,7){$\ROh \pi_F(C,\str_{F(C)})$}
\put(5,1.5){\vector(0,1){4.5}}
\put(7.5,6.5){\vector(3,-1){3.5}}
\put(7.5,1.5){\vector(3,2){3.5}}
\end{picture}
\eeq
so in any case, $\forall p \in \underline{F(C)}^+$, $p' = pu \in \underline{F(C')}^+$:
\beq
\setlength{\unitlength}{0.5cm}
\begin{picture}(13,8)(0,0)
\thicklines
\put(1,0){$\ROh \pi_F(C',p')$}
\put(9,1){$\phi_{C'}(p')$}
\put(12,4){$E$}
\put(9,7){$\phi_C(p)$}
\put(3,4){$u_*$}
\put(1,7){$\ROh \pi_F(C,p)$}
\put(5,1.5){\vector(0,1){4.5}}
\put(7.5,6.5){\vector(3,-1){3.5}}
\put(7.5,1.5){\vector(3,2){3.5}}
\end{picture}
\eeq

$\RMod$ being cocomplete so is $\cE = \RModCop$ as a functor category into a cocomplete category. Thus the functor $\ROh \pi_F$ has a colimit $lF = \text{colimit}(\int F \xrarr{\pi_F} \cC \xrarr{\ROh} \cE)$, and each cocone as above arises from the existence of a unique vertical arrow as in:
\beq
\setlength{\unitlength}{0.5cm}
\begin{picture}(15,9)(0,0)
\thicklines
\put(1,8){$\ROh \pi_F(C,p)$}
\put(4,4){$u_*$}
\put(1,1){$\ROh \pi_F(C',p')$}
\put(13.5,1){$E$}
\put(14,8){$l(F)$}
\put(5,2){\vector(0,1){5}}
\put(8,1){\vector(1,0){3}}
\put(6,7){\vector(3,-2){7}}
\multiput(6.7,2.8)(0.3,0.2){22}{\circle*{0.15}}
\put(13,7){\vector(1,1){0.5}}
\multiput(7.5,8)(0.5,0){11}{\line(1,0){0.3}}
\put(12.5,8){\vector(1,0){0.5}}
\multiput(14,7)(0,-1){5}{\line(0,-1){0.3}}
\put(14,3){\vector(0,-1){0.5}}
\end{picture}
\eeq
Hence:
\beq
\text{Nat}(F, r(E)) \simeq \Hom_{\cE}(l(F),E)
\eeq
is natural in $F$ and $E$, thus $l \dashv r$.
\end{proof}

\begin{Correp}
Any $R$-linear $F \in \RModCop$ is a colimit of representables.
\end{Correp}
\begin{proof}
For such an $F$, Yoneda states that $r(F) = \text{Nat}(\ROh, F) \simeq F$, thus $r \simeq id_{\cE}$, which implies that $l \simeq id_{\cE}$ since $l: \cE \rightleftarrows \cE: r$, and we conclude that:
\beq
F \simeq l(F) = \text{colimit}(\int F \xrarr{\pi_F} \cC \xrarr{\ROh} \cE)
\eeq
which means that any $R$-linear element of $\RModCop$ is a colimit of $\RMod$-valued representables.
\end{proof}

We now prove the main theorem. By Proposition \ref{rep} and its Corollary, given $\cC$ a small category, $R$ a commutative ring, $\RModCop$ has a small set of generators, again keeping in mind that $\RModCop$ denotes $R$-linear functors from $\cC^{op}$ to $\RMod$ where for $\cC$ we take the category of selections of objects and morphisms of $\cC$ endowed with its $R$-module structure. This is true in particular of categories of $R$-natural representations, thus $\RNatModCop$ has a small set of generators. $\sRNatModCop$ will be defined later as a subcategory thereof, so it has a small set of generators as well.\\

Conversely let $\cE$ be a category with small hom-sets and suppose it has a small set of objects that generates it. Let $\cC$ be the small full subcategory of $\cE$ whose objects are the elements of the small set of generators just discussed. Then $\forall E \in \Ob(\cE)$ we have epimorphic families:
\beq
\Ed(C_i) = \{ C_i \rarr E  \}
\eeq
for each object $C_i$ of $\cC$. $E$ is defined by any of these sets as well as the way those morphisms glue together. In order to do so one has to group morphisms together, an operation we denote by $+$. For $i$ fixed, denote by $x(C_i)$ the morphism $x: C_i \rarr E$. Then grouping two such morphisms together corresponds to considering:
\beq
\setlength{\unitlength}{0.5cm}
\begin{picture}(7,10)(0,-2)
\thinlines
\put(1,6){\oval(2,2)[b]}
\put(6,6){\oval(2,2)[b]}
\thicklines
\thicklines
\put(1,6){\oval(2,2)[t]}
\put(6,6){\oval(2,2)[t]}
\put(3.5,1){\oval(3,2)}
\put(0,4){\line(0,1){2}}
\qbezier(0,4)(2,3)(2,1)
\put(3.5,6){\oval(3,4)[b]}
\put(7,4){\line(0,1){2}}
\qbezier(7,4)(5,3)(5,1)
\put(0,6){$x(C_i)$}
\put(5,6){$y(C_i)$}
\put(1,-1.5){$x(C_i) + y(C_i)$}
\put(3.5,-0.5){\line(0,1){1.5}}
\put(3.5,1){\circle*{0.2}}
\end{picture} \label{Eop}
\eeq
In what follows we will focus on how morphisms of any given set $\Ed(C_i)$ collate together, and therefore we will omit mention of the object $C_i$. Note that considering two objects $x$ and $y$ is equivalent to selecting those objects. Thus one can work with selection maps denoted by the same letter denoting the selected objects. Picking two such selections as in \eqref{Eop} corresponds to having a coherent abelian operad being defined on $\cE$. Selections form a tensor category, which is further braided, symmetric, with unit $ = 0$, no selection, the inverse for $+$ being the removal of a selection. The addition $+$ is associative, commutative. Thus selection objects form a symmetric monoidal category with duals on which a coherent abelian operad is defined, in particular a $\mathbb{Z}$-module by Lemma \ref{RnCat}. Further that selection of objects is made from a collection of mathematical objects, so depending on such a collection selections acquire a color. Let $R$ denote the set of such colors. This can be made into a symmetric monoidal category with duals. We have a sub-coherent abelian operad $\cO_R \in \cO^{(2|1)}$ defined on $R$ as well. In the same manner that the objects $x$ have been replaced by selection objects, selections in $R$ are viewed as choices of such selections (to avoid talking about selections of selections). The unit for the addition in $\cO_R$ is $\eset$ = no selection, the dual for + of each such selection is the removal of that selection. The unit for the product $\cdot$ in $\cO_R$ is $\bland$. This is a non-coherent operation. Thus each $\Ed(C_i) = \{x: C_i \rarr E \}$ gives rise to an assembly category, hence a $R$-module by Lemma \ref{RnCat}. More precisely it is a $R$-natural module by construction.\\

All such associations $C_i \mapsto \Ed(C_i)$ assemble into a map $\Ed: \cC \rarr \RNatMod$. Further any map $f: C_i \rarr C_j$ as in:
\beq
\setlength{\unitlength}{0.5cm}
\begin{picture}(7,6)(0,0)
\thicklines
\put(0,5){$f: C_i$}
\put(7,5){$C_j$}
\put(7,0){$E$}
\put(8,3){$x$}
\put(3,5){\vector(1,0){3}}
\put(7.5,4){\vector(0,-1){3}}
\multiput(3,4)(0.2,-0.2){17}{\circle*{0.15}}
\put(6,1){\vector(1,-1){0.5}}
\end{picture}
\eeq
induces a morphism $\Ed(C_j) \rarr \Ed(C_i)$ so $\Ed$ yields an object of $\cC^{op} \rarr \RNatMod$. As argued in the proof of Giraud's theorem, $E$ as a standalone object is equivalently given by any of the sets $\Ed(C_i)$ and gluing, but since $\cE$ has a set of generators, all such sets should be considered, hence $E$ is equivalently given by $\coprod_{C \in \Ob(\cC)} \Ed(C)+\cup$, which projects down to $\coprod_{C \in \Ob(\cC)} \Ed(C) \simeq \Ed$, a functor $\cC^{op} \rarr \RNatMod$.\\

Thus we have a functor $\phi : \cE \rarr \RModCop$. Indeed given a category $\cE$ with small hom-sets, let $\epiplus \cE$ be the category whose objects are sets $\text{set}_E+\cup$ of epimorphic families $\Ed(X)+\cup$ of morphisms, $X \in \Ob(\cE)$, for each object $E$ of $\cE$, and whose morphisms are induced by morphisms in $\cE$. Hence:
\beq
\text{set}_E+\cup = \coprod_{X \in \Ob(\cE)} \Ed(X)+\cup \simeq E \nonumber
\eeq

If $E$ and $F$ are two objects of $\cE$ then a morphims $E \rarr F$ in $\cE$ clearly induces a morphism $\Ed(X) \rarr \Fd(X)$ for all $X$, hence a morphism $\text{set}_E+\cup \rarr \text{set}_F+\cup$ in $\epiplus \cE$. It is clear that $\cE \simeq \epiplus \cE$ if $\cE$ has a small set of generators. Further there is a forgetful functor:
\beq
\coprod_{X \in \Ob(\cE)} \Ed(X)+\cup \rarr \coprod_{X \in \Ob(\cE)} \Ed(X)
\eeq

The operation of collecting morphisms to be glued defines a $\mathbb{Z}$-module structure on each $\Ed(X)$, thus such a disjoint union $\coprod \Ed(X)$ forms an element $\RNatMod(-,E)$ of $\RNatMod^{\cE^{op}}$. Now note that if $\cE$ has a small set of generators then it is sufficient to consider the $X$'s above to being objects of $\cC$, hence we have a functor:
\begin{align}
\phi: \cE & \xrarr{\sim} \text{epi}+\cE \rarr \RNatModCop \nonumber \\
 E & \mapsto \phi(E) = \RNatOh_{E} = \RNatMod(-,E) \nonumber \\
 f:E \mapsto F & \mapsto \phi(f): \RNatOh_E \mapsto \RNatMod(-,fE)
\end{align}

We would like an equivalence. Observe that $R$ as a set is the set of choices of selections of maps $x: C_i \rarr E$ for some $E$, so given a $R$-natural module $M$, each element $x$ of $M$ is a map $C_i \rarr E_x$. However $F(C)$  being in $\RNatMod$ and given the definition of $R$, there is + on $F(C)$ corresponding to a gluing of images into a same object $E$, that is $\forall x \in F(C)$, $E_{C,x} = E_C$. For $C,C' \in \Ob(\cC)$ we may have $E_C \neq E_{C'}$. To have the same object we consider only those functors into $\RNatMod$ all of whose $R$-modules are morphisms into a same object $E$. Further we consider only those modules $F(C)$ that consist of morphisms from a same source. Denote by $\sRNatModCop$ this functor category. Further define an equivalence relation on $\sRNatModCop$ by $F \sim F'$ if and only if for any $C \in \Ob(\cC)$, $F(C) = \{D \rarr E_F \}$, $F'(C) = \{ D' \rarr E_{F'} \}$, $E_F = E_{F'}$. Denote by $[F]$ the equivalence class of $F$ in $\sRNatModCop / \sim$. $\Phi$ descends to $\phi: \cE \rarr \sRNatModCop / \sim$.\\

We define the following functor:
\begin{align}
\Pi_0 : \sRNatModCop /\sim & \rarr \cE  \nonumber \\
[F] & \mapsto E_F \nonumber \\
\psi: [F] \Lrarr [F'] &\mapsto \Pi_0(\psi)  = E_F \rarr E_{F'}
\end{align}
Then:
\begin{align}
\Pi_0 \circ \phi(E,E') &= \Pi_0 [\RNatOh_E[E']]  \nonumber \\
&= E_{\RNatOh_E}   \nonumber \\
&= E
\end{align}
On the other hand:
\begin{align}
\phi \circ \Pi_0(F) &= \phi(E_F) \nonumber \\
&= [\RNatOh_{E_F}] \nonumber \\
&= [F]
\end{align}
Hence we have an equivalence, which proves the theorem.$\Box$\\

Note that this theorem pushes the restrictions onto the $R$-modules and not so much onto the category $\cE$ itself. In the next section we will prove a more general theorem, akin to Giraud's theorem for sheaves of sets, linking categories satisfying Giraud's conditions to be a Grothendieck topos which are in addition enriched in assembly categories in a sense to be precised, and sheaves of $R$-modules.

\section{Sheaves of $R$-modules and Grothendieck topos enriched in assembly categories}
A quick fix to the problem encountered in the previous section is to have $\cE$ enriched in an assembly category. It becomes clear that one needs $\cE$ to be cocomplete. Giraud showed that any category satisfying his conditions was cocomplete. One could start with a cocomplete category without imposing his conditions. It becomes clear as one progresses through the proof of a general theorem involving categories of $R$-modules representations that one needs Giraud's conditions, first because we need a notion of gluing which translates on the $R$-module side as the existence of matching families for a Grothendieck topology, but in addition we will prove an equivalence derived from an adjunction which necessitates all of Giraud's conditions. Hence we arrive at the following general statement:
\begin{ultthm}
A category $\cE$ with small hom-sets and finite limits is equivalent to a category of representations $\RModCop$ for some commutative ring $R$ if and only if $\cE$ satisfies Giraud's axioms for being a Grothendieck topos and it is enriched in a strict assembly category parametrized by a $R$-module.
\end{ultthm}
The proof is surprisingly similar to that for showing that categories satisfying Giraud's axioms are Grothendieck topos. Thus we will not repeat the proof that is classic in its entirety. Rather we will mention the generalization of Giraud's results to our setting to show the path to follow, if the proofs are identical we will just say for the sake of stating them that they are the same, and whenever one has to be detailed in the case of $R$-modules we will show the details. Anytime we need the conditions of Giraud's theorem (i)-(v) those are found in Theorem \ref{Gir}. \\

We first show, using the notations of \cite{ML}, that if $\cC$ denotes the full subcategory of $\cE$ whose objects are the generators of $\cE$, that for the inclusion functor $A: \cC \rarr \cE$ we have a Hom-tensor adjunction:
\beq
- \otimes_{\cC}A: \RModCop \rightleftarrows \cE: \ROh = \RMod(A,-) \nonumber
\eeq
Since $\cE$ is enriched in a strict assembly category parametrized by some $R$-module, there is a forgetful map $\Hom_{\cE}(E,E') \xrarr{\sim} \RMod(E,E')$ that is an isomorphism, whence the existence of a functor $\RMod(A,-)$ as shown above. To construct such an equivalence we start from the functors on the right hand side. For $\cC$ a small category, $\cE$ a cocomplete category, $A: \cC \rarr \cE$ a functor, we have:
\begin{align}
R: \cE & \rarr \RModCop \nonumber \\
E & \mapsto R(E): C \mapsto R(E)(C) = \ROh_E(C) = \RMod(C,E)
\end{align}
has a left adjoint:
\beq
L: \RModCop \rightleftarrows \cE:R \nonumber
\eeq
given by:
\beq
L_A(F) = \lim_{\rarr}(\int F \xrarr{\pi_F} \cC \xrarr{A} \cE) \nonumber
\eeq
This we will use to construct $- \otimes_{\cC}A$. Thus we need $\cE$ to be cocomplete. For this to happen we use App.2.1 \& 2.2 of \cite{ML} for the case of $\RMod$, which gives conditions for $\cE$ to be cocomplete. These two results carry over to the $\RMod$ case since those proofs are independent of any enrichment. Lemma App.2.2 of \cite{ML} uses the fact that $\cE$ has finite limits and (iii), that any equivalence relation has a quotient, and that any coequalizer is stable under pullback by (ii) and (iv). Proposition App.2.1 then uses the existence of coequalizers for equivalence relations (iii) and the fact that coproducts and epis are stable under pullback. Thus we impose Giraud's conditions for cocompleteness, that is existence of small hom-sets, finite limits and (i)-(iv). We now prove the above result about the existence of a left adjoint to $R$. Here $\RMod$ is for some fixed commutative ring $R$ with a unit. $\cE$ is assumed to be enriched in a $R$-colored assembly category. For $F \in \RModCop$, any natural transformation $\phi: F \Lrarr R(E)$ is given by a family of $R$-module homomorphisms:
\beq
\{\phi_C: F(C) \rarr \ROh_E(C) \}
\eeq
natural in $C$. For $u:C' \rarr C$ in $\cE$ we have a commutative diagram:
\beq
\begin{CD}
F(C) @>\phi_C>> \ROh_E(C) \\
@VF(u)VV @VVA(u)^*V \\
F(C') @>>\phi_{C'}> \ROh_E(C')
\end{CD} \label{ONE}
\eeq
element-wise:
\beq
\begin{CD}
F(C) \ni f @>\phi_C>> \{\phi_C(f):A(C) \rarr E \}=\RMod(C,E)_f \\
@VF(u)VV @VVV \\
F(C') \ni f' @>>\phi_{C'}> \{ \phi_{C'}(f'):A(C') \rarr E \} = \RMod(C',E)_{f'}
\end{CD} \label{ceedee}
\eeq
Regarding the algebraic structures, if $\str_M$ denotes the $R$-module structure on a $R$-module $M$ then:
\beq
\begin{CD}
\str_{F(C)} @>\phi_C>> \str_{\ROh_E(C)} \\
@VF(u)VV @VVV \\
\str_{F(C')} @>>\phi_{C'}> \str_{\ROh_E(C')}
\end{CD} \label{STR}
\eeq

Hence \eqref{ONE} is fully described by \eqref{ceedee} and \eqref{STR}. With $F(u)f = f \cdot u = f'$ element-wise, we have:
\beq
\setlength{\unitlength}{0.5cm}
\begin{picture}(13,8)(0,0)
\thicklines
\put(3,0){$A \pi_F(C',f')$}
\put(9,1){$\phi_{C'}(f')$}
\put(12,4){$E$}
\put(9,7){$\phi_C(f)$}
\put(3,4){$u_*$}
\put(3,7){$A\pi_F(C,f)$}
\put(5,1.5){\vector(0,1){4.5}}
\put(7.5,6.5){\vector(3,-1){3.5}}
\put(7.5,1.5){\vector(3,2){3.5}}
\end{picture}
\eeq
Moreover each such map $A\pi_F(C,f) \rarr E$ and $A \pi_F(C',f') \rarr E$ is an element of $\ROh_E(C)$ and $\ROh_E(C')$ respectively, and those have a $R$-module structure compatible with those of $F(C)$ and $F(C')$ respectively via $\phi_C(\str_{F(C)}) = \str_{\ROh_E(C)}$ and $\phi_{C'}(\str_{F(C')}) = \str_{\ROh_E(C')}$. We can represent this as a decoration of each generic element $A\pi_F(C) \rarr E$ of $\ROh_E(C)$ as follows:
\beq
\setlength{\unitlength}{0.5cm}
\begin{picture}(13,8)(0,0)
\thicklines
\put(2,0){$A \pi_F(C',\str_{F(C')})$}
\put(9,1){$\phi_{C'}(\str_{F(C')})$}
\put(12,4){$E$}
\put(9,7){$\phi_C(\str_{F(C)})$}
\put(3,4){$u_*$}
\put(2,7){$A\pi_F(C,\str_{F(C)})$}
\put(5,1.5){\vector(0,1){4.5}}
\put(7.5,6.5){\vector(3,-1){3.5}}
\put(7.5,1.5){\vector(3,2){3.5}}
\end{picture}
\eeq
by regarding the algebraic structure $\str_{F(C)}$ as an element of $\underline{F(C)}^+ \simeq F(C)$. Hence we have a cocone on $A\pi_F$ over $E \in \cE$ since the morphisms involved are elements of some $R$-module, $\cE$ being strict, and $\cE$ being cocomplete, there is a unique arrow $LF \rarr E$ in $\cE$ giving rise to such a cocone:
\beq
\setlength{\unitlength}{0.5cm}
\begin{picture}(15,9)(0,0)
\thicklines
\put(2.5,8){$A \pi_F(C,f)$}
\put(4,4){$u_*$}
\put(2,1){$A \pi_F(C',f')$}
\put(13.5,1){$E$}
\put(14,8){$L(F)$}
\put(5,2){\vector(0,1){5}}
\put(7,1){\vector(1,0){5}}
\put(10,1.5){\circle*{0.3}}
\put(6,7){\vector(3,-2){7}}
\put(7,6){\circle*{0.3}}
\multiput(6.7,2.8)(0.3,0.2){22}{\circle*{0.15}}
\put(7.5,4){\circle*{0.3}}
\put(13,7){\vector(1,1){0.5}}
\multiput(7,8)(0.5,0){12}{\line(1,0){0.3}}
\put(12.5,8){\vector(1,0){0.5}}
\put(10,8.5){\circle*{0.3}}
\multiput(14,7)(0,-1){5}{\line(0,-1){0.3}}
\put(14,3){\vector(0,-1){0.5}}
\end{picture}
\eeq
where we have represented by a fat dot the fact that those are elements of a $R$-module. Thus all $\phi_C: F(C) \rarr \ROh_E(C)$ collectively are equivalently given by a unique map $LF \rarr E$ in $\cE$ and all such maps $\phi_C$ combine to give a $F \Lrarr R(E)$. Thus:
\beq
\text{Nat}(F,R(E)) \simeq \Hom_{\cE}(LF,E)
\eeq
or equivalently:
\beq
\Hom_{\RModCop}(F, \ROh_E) \simeq \Hom_{\cE}(LF,E)
\eeq
that is $L \dashv R$. \\

Note that $\lim_{\rarr}(\int F \rarr \cC \rarr \cE)$ can be expressed as a tensor product as done in \cite{ML} for the case $F \in (\text{Sets})^{\cC^{op}}$. This carries over to the case of $F \in \RModCop$ and $L(F)$ can be presented as a coequalizer:
\beq
\coprod_{\substack{C, f \in F(C) \\ u:C' \rarr C}} A(C') \rrarrop^{\theta}_{\tau} \coprod_{C, f \in F(C)} A(C) \xrarr{\phi} L(F) = F \otimes_{\cC}A
\eeq
where:
\begin{align}
\theta(A(C')_{C,f,u}) &= A(C')_{f'=fu} \nonumber \\
\tau(A(C')_{C,f,u}) &=A(C)_f
\end{align}
which includes in particular the special case $f = \str_{F(C)}$ and $f' = \str_{F(C')}$ since $F(C) \simeq \underline{F(C)}^+$. Thus $L: \RModCop \rlarr \cE: \ROh$ can be denoted by:
\beq
-\otimes_{\cC} A: \RModCop \rlarr \cE: \RMod(A,-)
\eeq
We now investigate the unit and counit of this adjunction:
\begin{align}
\eta: 1_{\RModCop} & \Lrarr \RMod(A, -\otimes_{\cC}A) \nonumber \\
F & \rarr \RMod(A, F \otimes_{\cC}A)
\end{align}
with components:
\begin{align}
\eta_{F,C}: F(C) &\rarr \RMod(A(C), F \otimes_{\cC}A) \nonumber \\
f &\mapsto f\otimes - : A(C) \rarr F\otimes_{\cC}A
\end{align}
and the counit:
\beq
\epsilon_E: \RMod(A, E) \otimes_{\cC}A  \rarr E
\eeq
as the vertical map shown below:
\beq
\setlength{\unitlength}{0.5cm}
\begin{picture}(25,7)(0,0)
\thicklines
\put(0,6){$\coprod_{\substack{u:C' \rarr C \\ x: A(C) \rarr E}}A(C')$}
\put(7,6){\vector(1,0){3}}
\put(7,6.5){\vector(1,0){3}}
\put(8,7){$\theta$}
\put(8,5.5){$\tau$}
\put(11,6){$\coprod_{\substack{C \in \Ob(\cC) \\ x: A(C) \rarr E}}A(C)$}
\put(17,6){\vector(1,0){4}}
\put(19,6.5){$\pi$}
\put(22,6){$\RMod(A,E) \otimes_{\cC} A$}
\multiput(25,5)(0,-0.5){8}{\line(0,-1){0.3}}
\put(25,1){\vector(0,-1){0.5}}
\put(24.5,-1){$E$}
\put(17,5){\vector(4,-3){6.5}}
\put(19,2){$X$}
\end{picture} \nonumber
\eeq
with
\beq
X = \coprod x = \coprod_{\substack{C \in \Ob(\cC) \\ x:A(C) \rarr E}} \{ A(C) \xrarr{x} E \}
\eeq
Note the $R$-module structures on $\RMod(A(C),E)$ and $\RMod(A(C'),E)$ are compatible with $u$.
\begin{unit}
$\epsilon_E$ is an isomorphism for all $E \in \Ob(\cE)$.
\end{unit}
\begin{proof}
Since $\cC$ generates $\cE$, there is an epi:
\beq
X: \coprod_{\substack{C \in \Ob(\cC) \\ x: C \rarr E}} C \rarr E
\eeq
Note that for that coproduct to be indexed by a small set, in particular we need $\cE$ to have small hom-sets. Now we argue as in \cite{ML} that $X$ fits into a coequalizer diagram:
\beq
\coprod_s D \rrarrop_{\beta}^{\alpha} \coprod_{\substack{C \in \Ob(\cC) \\ x: C \rarr E}} C \xrarr{X} E \label{DCE}
\eeq
where $s$ indexes commutative diagrams in $\cE$ of the form:
\beq
\begin{CD}
D @>h>> C \\
@VkVV @VVxV \\
C @>>x'> E
\end{CD}
\eeq
where $D \in \Ob(\cC)$ and $\alpha$ and $\beta$ have $h$ and $k$ for respective components. To prove the existence of such a coequalizer as in App.3.1 of \cite{ML}, one uses the fact that epis are coequalizers and that coproducts are stable under pullback. Having the coequalizer \eqref{DCE}, we now prove that for all $B \in \Ob(\cE)$, $Y: \coprod_{C,x} C \rarr B$ coequalizes $\alpha$ and $\beta$ if and only if it coequalizes $\theta$ and $\tau$. This would show:
\beq
\RMod(A,E) \otimes_{\cC}A \simeq B \simeq E
\eeq
The proof in our case would be exactly similar to that of App.3.1 of \cite{ML}. But this means $\epsilon_E$ is an isomorphism.
\end{proof}
Now to prove $\eta$ is an isomorphism we need $\RMod(A,E)$ to be a sheaf, so we need a topology on $\cC$. Since $\cC$ generates $\cE$, we take covers to be epimorphic families. We have proved already that this defines a Grothendieck topology $J$ on $\cC$. We now prove that for all $E \in \Ob(\cE)$, $\RMod(-,E): \cC^{op} \rarr \RMod$ is a sheaf for $J$ on $\cC$:
\begin{Esheaf}
For all $E \in \Ob(\cE)$, $\RMod(-,E)$ is a sheaf on $(\cC,J)$.
\end{Esheaf}
\begin{proof}
Exactly as done in \cite{ML}, let $C \in \Ob(\cC)$, $S$ a covering sieve of $C$, $p_S: \coprod_{g \in S} D \xrarr{g} C$, an epi in $\cE$, hence a coequalizer. Using the fact that $\cE$ preserves coproducts under pullback and the fact that $\cC$ generates $\cE$, one obtains a coequalizing diagram:
\beq
\coprod B \rrarr \coprod_{g \in S} D \rarr C \label{coeq}
\eeq
We claim this is mapped by $\RMod(-,E)$ into an equalizer. For this to hold $\RMod(-,E)$ must map colimits to limits. We prove $\RMod(E,-)$ preserves limits, this will prove $\RMod(-,E)$ maps colimits to limits by duality. Let $J$ be an indexing category, $F:J \rarr \cC$ be a functor with a limit:
\beq
\setlength{\unitlength}{0.5cm}
\begin{picture}(7,8)(0,0)
\thicklines
\put(0,4){$\lim F$}
\put(3,1){$\mu_j$}
\put(5.5,1){$F_j$}
\put(6.5,4){$f_{ji}$}
\put(5.5,7){$F_i$}
\put(2,6){$\mu_i$}
\put(2,3){\vector(3,-1){3}}
\put(6,2){\vector(0,1){4}}
\put(2,5){\vector(3,2){3}}
\end{picture}
\eeq
Since $\cC \subset \cE$, $\cE$ being enriched in a strict assembly category, then morphisms displayed in this cone are elements of a $R$-module. This cone maps under $\RMod(E,-)$ to:
\beq
\setlength{\unitlength}{0.5cm}
\begin{picture}(7,8)(0,0)
\thicklines
\put(-1.5,4){$\RMod(E,\lim F)$}
\put(5.5,1){$\RMod(E,F_j)$}
\put(5.5,7){$\RMod(E,F_i)$}
\put(2,3){\vector(3,-1){3}}
\put(8,2){\vector(0,1){4}}
\put(2,5){\vector(3,2){3}}
\end{picture} \label{limcone}
\eeq
in $\RMod$. We show this is a limiting cone: for any cone from a $R$-module $M$:
\beq
\setlength{\unitlength}{0.5cm}
\begin{picture}(17,11)(0,0)
\thicklines
\put(1,1){$M$}
\put(11,1){$\RMod(E,F_j)$}
\put(11,9){$\RMod(E,F_i)$}
\put(0,9){$\RMod(E,\lim F)$}
\put(6,2){$\lambda_j$}
\put(4,4){$\lambda_i$}
\put(6,7){$\mu_{j*}$}
\put(8,10){$\mu_{i*}$}
\multiput(3,1)(0.5,0){14}{\line(1,0){0.3}}
\put(9.5,1){\vector(1,0){0.5}}
\multiput(3,2)(0.8,0.6){11}{\circle*{0.2}}
\put(11,8){\vector(1,1){0.5}}
\put(4,8){\vector(4,-3){8}}
\put(13,3){\vector(0,1){5}}
\put(7,9){\vector(1,0){3}}
\end{picture}
\eeq
$a \in M$ yields a cone:
\beq
\setlength{\unitlength}{0.5cm}
\begin{picture}(7,8)(0,0)
\thicklines
\put(1,4){$E$}
\put(3,1){$\lambda_j a$}
\put(5.5,1){$F_j$}
\put(5.5,7){$F_i$}
\put(2,6){$\lambda_i a$}
\put(2,3){\vector(3,-1){3}}
\put(6,2){\vector(0,1){4}}
\put(2,5){\vector(3,2){3}}
\end{picture}
\eeq
In particular we have a structural cone for $a = \str_M$:
\beq
\setlength{\unitlength}{0.5cm}
\begin{picture}(7,8)(0,0)
\thicklines
\put(1,4){$E$}
\put(2,1){$\lambda_j \str_M$}
\put(5.5,1){$F_j$}
\put(5.5,7){$F_i$}
\put(1,6.5){$\lambda_i \str_M$}
\put(2,3){\vector(3,-1){3}}
\put(6,2){\vector(0,1){4}}
\put(2,5){\vector(3,2){3}}
\end{picture}
\eeq
and by universality there is a unique $\lambda_a \in \Mor(\cC) \subset \Mor(\cE)$ for $a \in \underline{M}^+$ such that:
\beq
\setlength{\unitlength}{0.5cm}
\begin{picture}(17,11)(0,0)
\thicklines
\put(1.5,1){$E$}
\put(12.5,1){$F_j$}
\put(12.5,9){$F_i$}
\put(1,9){$\lim F$}
\put(6,7){$\mu_j$}
\put(7,10){$\mu_i$}
\multiput(3,1)(0.5,0){18}{\line(1,0){0.3}}
\put(11.5,1){\vector(1,0){0.5}}
\multiput(3,2)(0.8,0.6){11}{\circle*{0.2}}
\put(11,8){\vector(1,1){0.5}}
\put(4,8){\vector(4,-3){8}}
\put(13,3){\vector(0,1){5}}
\put(3.5,9){\vector(1,0){7}}
\put(2,3){\vector(0,1){5}}
\put(1,5){$\lambda_a$}
\end{picture}
\eeq
with $\mu_i \lambda_a = \lambda_i a$. Let $\Lambda(a) = \lambda_a$ for $a \in \underline{M}^+$. Then we get a map $\Lambda: M \rarr \RMod(E, \lim F)$ such that $\mu_{i*} \Lambda = \lambda_i$ for any $i \in J$ and $\Lambda$ is unique by construction, hence \eqref{limcone} is a limiting cone. Thus by applying $\RMod(-,E)$ to \eqref{coeq} we get an equalizer:
\beq
\prod_B \RMod(B,E) \leftleftarrows \prod_{g \in S} \RMod(D,E) \leftarrow \RMod(C,E) \nonumber
\eeq
in $\RMod$. Thus $C \mapsto \RMod(C,E)$ satisfies the sheaf condition for $S$, for all covering sieves $S$ in $J$, which implies that $C \mapsto \RMod(C,E)$ is a $J$-sheaf.
\end{proof}

Denote by $\ShRCJ$ the category of sheaves on $(\cC,J)$ valued in $\RMod$.\\

We now prove the unit $\eta: 1_{\RModCop} \Lrarr \RMod(A, -\otimes_{\cC}A)$ gives isomorphisms $\eta_F: F \rarr \RMod(A, F \otimes_{\cC}A)$ for $F \in \ShRCJ \subset \RModCop$ with components
\begin{align}
\eta_{F,C}: F(C) & \rarr \RMod(A(C), F \otimes_{\cC} A) \nonumber \\
f &\mapsto f\otimes - : A(C) \mapsto F \otimes_{\cC}A
\end{align}
Since any $F \in \ShRCJ$ is a colimit of representables, we first prove $\eta$ is an isomorphism on representables $a \RMod(-,E)$ where $a:\RModCop \rarr \ShRCJ$ is the associated sheaf functor. To do this we will need to prove that representables are units for the tensor product. To generalize the isomorphism to one on $\ShRCJ$ we will need $\RMod(A, - \otimes_{\cC}A)$ to preserve colimits. Since $F \mapsto F \otimes_{\cC}A$ is a left adjoint it preserves colimits. If $\RMod(A,-)$ preserves colimits as well, then so will $\RMod(A, -\otimes_{\cC}A)$, so we prove $\RMod(A,-)$ preserves colimits as well. This will prove that $\eta$ restricts to an isomorphism on $\ShRCJ$. First:
\begin{repunit}
Representables are units for the tensor product.
\end{repunit}
\begin{proof}
We want to prove that the ascending map on the right below is an isomorphism, and this will prove the result:
\beq
\setlength{\unitlength}{0.5cm}
\begin{picture}(25,9)(1,-1)
\thicklines
\put(0,6){$\coprod_{\substack{C \in \Ob(\cC) \\ \psi \in \RMod(C,E) \\ u:C' \rarr C}}$}
\put(4.5,6){$A(C')$}
\put(7,6){\vector(1,0){3}}
\put(7,6.5){\vector(1,0){3}}
\put(8,7){$\theta$}
\put(8,5.5){$\tau$}
\put(11,6){$\coprod_{\substack{C \in \Ob(\cC) \\ \psi \in \RMod(C,E)}}$}
\put(15.5,6){$A(C)$}
\put(17.5,6){\vector(1,0){3.5}}
\put(19,6.5){$\phi$}
\put(22,6.5){$L(\ROh_E)$}
\put(24,5.5){$\|$}
\put(22,4.5){$ \ROh_E \otimes_{\cC}A$}
\put(24,-1){$A(E)$}
\put(17,5){\vector(4,-3){6.5}}
\put(19,2){$\Psi$}
\put(25,0){\vector(0,1){3.5}}
\end{picture} \nonumber
\eeq
where $\Psi = \coprod \psi$, the composition at the top is the coequalizer defining the left adjoint $L$, the bottom composition existing by construction. To write $\Psi \theta = \Psi \tau$ means
\beq
\coprod_{C'}\psi |_{\theta A(C')}  = \coprod_{C'}\psi |_{\tau A(C')}
\eeq
on components, with obvious notations, and with $\psi |_{\theta A(C')} =  C' \xrarr{u} C \xrarr{\psi} E$ and $\psi |_{\tau A(C')} = C \xrarr{\psi} E$. Now if $\Psi = \coprod \psi$, then $\Psi$ coequalizes $\theta$ and $\tau$ by construction, hence $A(E) \xrarr{\sim} \ROh_E \otimes_{\cC}A$ for any $E$.
\end{proof}
\begin{etaiso}
$\eta$ is an isomorphism on representables.
\end{etaiso}
\begin{proof}
We have just proved that for any $C \in \Ob(\cC)$, then $A(C) \xrarr{\sim} \ROh_C \otimes_{\cC} A$. With $\ROh_C(D) = \RMod(D,C)$, then $\ROh_C = \RMod(A, A(C))$, thus:
\beq
\setlength{\unitlength}{0.5cm}
\begin{picture}(22,6)(0,0)
\thicklines
\put(0,4){$\ROh_C$}
\put(6,0){$\RMod(A,A(C))$}
\put(14,4){$\RMod(A, \ROh_C \otimes_{\cC}A)$}
\put(3,3){\vector(1,-1){2}}
\put(3,1.5){$\sim$}
\put(12,1){\vector(1,1){2}}
\put(13,1.5){$\sim$}
\multiput(4,4)(0.5,0){18}{\line(1,0){0.3}}
\put(12.5,4){\vector(1,0){0.5}}
\put(7,4.5){$\sim$}
\end{picture}
\eeq
Now $\eta_{\ROh_C, D}: \ROh_C(D) \rarr \RMod(A(D), \ROh_C \otimes_{\cC}A)$ so this diagram commutes, hence $\eta$ is an isomorphism on representables.
\end{proof}
It remains to prove that:
\begin{preco}
$\RMod(A,-): \cE \rarr \ShRCJ$ preserves colimits.
\end{preco}
\begin{proof}
We follow the proof of \cite{ML} for the case of sheaves of sets modified to our setting. In proving this lemma we will need that $\RMod(A,-)$ preserves epis. In proving such a fact we will need to prove that for an epi $u: E' \rarr E$, $u_*: \RMod(-,E') \rarr \RMod(-,E)$ is locally surjective, as well as the fact that $u_*$ is an epi in $\ShRCJ$ as a consequence of the local surjectivity of $u_*$, so we consider this first:
\begin{locsurj}
A morphism $\phi: F \rarr G$ of sheaves of $R$-modules is an epi in $\ShRCJ$ if for all $C \in \Ob(\cC)$, for all $y \in G(C)$, there is a cover $S$ of $C$ such that for all $f: D \rarr C$ in $S$, $y \cdot f$ is in the image of $\phi_D: F(D) \rarr G(D)$, which is referred to as $\phi$ being locally surjective.
\end{locsurj}
\begin{proof}
This is proved exactly as in III.7.5 of \cite{ML}.
\end{proof}

Now we need that $\RMod(A,-): \cE \rarr \ShRCJ$ preserves epis:
\begin{presepi}
$\RMod(A, -): \cE \rarr \ShRCJ$ preserves epis.
\end{presepi}
\begin{proof}
Exactly as in the proof of App.3.3 of \cite{ML}. One starts from an epi $u:E' \rarr E$ in $\cE$. Using that epis are stable under pullback and the fact that $\cE$ is enriched in a strict assembly category parametrized by a $R$-module, pullback diagrams as in App.3.3 in $\cE$ project down to pullbacks involving $R$-homomorphisms which lead to $u_*: \ROh_{E'} \rarr \ROh_E$ being locally surjective, hence epi by the preceding lemma.
\end{proof}

Finally we can prove that $\RMod(A,-)$ preserves colimits. This uses the fact that for $\cE$ satisfying (i)-(iv), for any category $\cF$, then if $F: \cE \rarr \cF$ preserves coproducts, finite limits and coequalizers of equivalence relations, then $F$ preserves all coequalizers, and consequently all colimits, the proof of which is done in App.2.3 of \cite{ML} and carries over to the case of $\RMod(A,-): \cE \rarr \ShRCJ$. Using the fact that epis are coequalizers in $\ShRCJ$, (iii), and the fact that $\RMod(A,-)$ is left exact since $\Hom_{\cE}(A,-)$ is, then $\RMod(A,-)$ preserves coequalizers of equivalence relations. Thus it remains to prove that $\RMod(A,-)$ preserves small coproducts.\\

One important fact that was true for sheaves of sets is that for any small family $\{F_{\alpha} \}$ of sheaves in $\ShCJ$:
\beq
\coprod_{\substack{\alpha \\ \text{sheaves}}} F_{\alpha} \simeq a \big( \coprod_{\substack{\alpha \\ \text{sets}}} F_{\alpha} \big)
\eeq
where $a: \SetCop \rarr \ShCJ$ is the sheafification functor. This follows from the more general fact that:
\beq
a(\lim_{\rarr} i F_{\alpha}) \simeq \lim_{\rarr} ai F_{\alpha} \simeq \lim_{\rarr} F_{\alpha}
\eeq
where $i: \ShCJ \rarr \SetCop$ with $a \dashv i$. Thus we would like to prove for sheaves of $R$-modules that the inclusion functor $i: \ShRCJ \rarr \RModCop$ has a left adjoint $a: \RModCop \rarr \ShRCJ$ as well as $a \circ i: \ShRCJ \rarr \ShRCJ$ is isomorphic to the identity, in which case we would have that for any small family of sheaves of $R$-modules $\{ F_{\alpha} \}$:
\beq
a \coprod_{\alpha} i F_{\alpha} \simeq \coprod_{\alpha} ai F_{\alpha} \simeq \coprod_{\alpha} F_{\alpha}
\eeq
\begin{Ladj}
The inclusion functor $i: \ShRCJ \rarr \RModCop$ has a left adjoint $a: \RModCop \rarr \ShRCJ$
\end{Ladj}
\begin{proof}
Proof as in III.5.1 of \cite{ML}. One defines $a(F) = (F^+)^+$ where $F^+(C) = \colim_{S \in J(C)} \text{match}(S,F)$ where $\text{match}(S,F)$ denotes the set of matching families for $S \in J(C)$. Note that $+$ is a functor of presheaves of $R$-modules since for any map $\psi: F \rarr G$ of elements of $\RModCop$, we have an induced map $\psi^+: F^+ \rarr G^+$. We prove for all $F \in \RModCop$, $F^+$ is a separated sheaf, meaning it has at most one amalgamation, and if $F$ is separated, then $F^+$ is a sheaf, making $a(F) = (F^+)^+$ a sheaf $\forall F \in \RModCop$. That $F^+$ is separated for $F \in \RModCop$ is proved exactly as in III.3.4 of \cite{ML}. That $F^+$ is a sheaf if $F$ is separated is proved as in III.3.5 of \cite{ML}. Thus $a(F) = (F^+)^+$ defines a functor $\RModCop \rarr \ShRCJ$. It turns out that in addition to this functor, there is a map of presheaves:
\begin{align}
\eta: F &\rarr F^+ \nonumber \\
x \in F(C) & \mapsto \eta_C(x) = [\{x \cdot f \, | \, f \in \text{max sieve on }C \}]
\end{align}
where the brackets denote equivalence classes of matching families as defined in \S III.5 of \cite{ML}. We need the following fact: that for $F \in \RModCop$, $G \in \ShRCJ$, for any $\phi: F \rarr G$, there is a unique factorization through $\eta$ as in:
\beq
\setlength{\unitlength}{0.5cm}
\begin{picture}(6,7)(0,0)
\thicklines
\put(0,6){$F$}
\put(5,6){$F^+$}
\put(4.7,1){$G$}
\put(1,6){\vector(1,0){3}}
\put(2,6.5){$\eta$}
\put(1,5){\vector(1,-1){3}}
\put(2,3){$\phi$}
\multiput(5,5)(0,-0.5){5}{\line(0,-1){0.3}}
\put(5,2.5){\vector(0,-1){0.5}}
\end{picture}
\eeq
and this is proved exactly as in III.5.3. Now $F \xrarr{\eta} F^+ \xrarr{\eta} F^{++} = a(F)$ and $F \xrarr{a} a(F)$ with $a(F)$ a sheaf, so there is a unique map $\tilde{a}$ as shown below:
\beq
\setlength{\unitlength}{0.5cm}
\begin{picture}(6,7)(0,0)
\thicklines
\put(0,6){$F$}
\put(5,6){$F^+$}
\put(4.5,1){$a(F)$}
\put(1,6){\vector(1,0){3}}
\put(2,6.5){$\eta$}
\put(1,5){\vector(1,-1){3}}
\put(2,3){$a$}
\multiput(5,5)(0,-0.5){5}{\line(0,-1){0.3}}
\put(5,2.5){\vector(0,-1){0.5}}
\put(6,4){$\tilde{a}$}
\end{picture}
\eeq
Now given $\tilde{a}$, there is a unique $\tilde{\tilde{a}}$ as shown below:
\beq
\setlength{\unitlength}{0.5cm}
\begin{picture}(14,7)(0,0)
\thicklines
\put(0,6){$F$}
\put(5,6){$F^+$}
\put(4.5,1){$a(F)$}
\put(1,6){\vector(1,0){3}}
\put(2,6.5){$\eta$}
\put(1,5){\vector(1,-1){3}}
\put(2,3){$a$}
\multiput(5,5)(0,-0.5){5}{\line(0,-1){0.3}}
\put(5,2.5){\vector(0,-1){0.5}}
\put(6,4){$\tilde{a}$}
\put(6,6){\vector(1,0){3}}
\put(7,6.5){$\eta$}
\put(10,6){$F^{++} = a(F)$}
\multiput(11.5,5)(-0.5,-0.3){10}{\circle*{0.15}}
\put(6.7,2.2){\vector(-3,-2){0.5}}
\put(10,3){$\tilde{\tilde{a}}$}
\end{picture}
\eeq
but $id$ is one such map satisfying this, so by uniqueness $\tilde{\tilde{a}} = id$, hence $a \dashv i$ by universality. Further $a(F) = ia(F)$ so $\eta^2$ is the unit of the adjunction $a \dashv i$.
\end{proof}
We also need the following fact:
\begin{ai}
$a \circ i: \ShRCJ \rarr \ShRCJ$ is naturally isomorphic to the identity functor.
\end{ai}
\begin{proof}
This follows from III.5.2: $F$ is a sheaf if and only if $\eta: F \rarr F^+$ is an isomorphism. Suppose $F$ is a sheaf of $R$-modules. Then $F = i(F) \xrarr{\eta} F^+$, so $i(F) \simeq F^+$, hence $ai(F) \simeq a(F^+) = F^+ = F$ for all $F$, hence $ai$ is isomorphic to the identity functor.
\end{proof}
Thus what we have obtained thus far is:
\begin{align}
\coprod_{\substack{\alpha \\ \ShRCJ}} F_{\alpha} &\simeq \coprod_{\substack{\alpha \\ \ShRCJ}} ai F_{\alpha} \nonumber \\
& \simeq a \big( \coprod_{\substack{\alpha \\ \ShRCJ}} iF_{\alpha} \big) \nonumber \\
& \simeq a \big( \coprod_{\substack{\alpha \\ \RModCop}} F_{\alpha} \big)
\end{align}
To prove that $\RMod(A,-)$ preserves small coproducts, let $\coprod_{\alpha} E_{\alpha}$ be one such small coproduct in $\cE$, with the coproduct inclusions $i_{\alpha}: E_{\alpha} \rarr \coprod_{\alpha} E_{\alpha}$. Then we have induced morphisms:
\beq
\setlength{\unitlength}{0.5cm}
\begin{picture}(21,14)(0,0)
\thicklines
\put(-1,13){$\coprod_{\substack{\alpha \\ \RModCop}} \RMod(A, E_{\alpha})$}
\put(14,13){$\RMod(A, \coprod_{\alpha}E_{\alpha})$}
\put(-2,6){$ia \coprod_{\substack{\alpha \\ \RModCop}} \RMod(A, E_{\alpha})$}
\put(4,4){$\simeq$}
\put(1,2){$i \coprod_{\alpha} \RMod(A, E_{\alpha})$}
\put(13,6){$a\RMod(A, \coprod_{\alpha}E_{\alpha})$}
\put(10,13){\vector(1,0){3}}
\put(11,13.5){$\phi$}
\put(4,12){\vector(0,-1){4}}
\put(5,9){$\eta = ia$}
\put(17,12){\vector(0,-1){4}}
\put(18,10){$\eta = ia$}
\put(10,3){\vector(3,1){5.5}}
\put(14,3){$a\phi$}
\end{picture}
\eeq
We want $a\phi$ to be an isomorphism in $\ShRCJ$. That $a\phi$ is epis is not difficult to prove for sheaves of $R$-modules and goes exactly as in App.3.4. This uses a variant of III.7.6 of \cite{ML} in our case: a morphism of presheaves $\psi: F \rarr G$ is such that $a\psi: aF \rarr aG$ is epi if $\psi$ is locally surjective. The proof of such a fact goes as in the proof of III.7.6 and uses the fact that $\eta: F \rarr aF$ is universal and natural. The rest of the proof that $a\phi$ is epi uses the fact that coproducts are disjoint in $\cE$ and is otherwise the same, except for the case where the element $r:C \rarr \coprod_{\alpha}E_{\alpha}$ to be picked is the $R$-module structure. Let $M = \RMod(C, \coprod_{\alpha}E_{\alpha})$, then let $r = \str_M$. For any $\alpha$ there is a pullback with $N_{\alpha} = \RMod(P_{\alpha}, E_{\alpha})$:
\beq
\begin{CD}
B @>>> P_{\alpha} @>\str_{N_{\alpha}}>> E_{\alpha} \\
@. @VVj_{\alpha}V @VVi_{\alpha}V \\
 @. C @>>\str_M> \coprod_{\alpha}E_{\alpha}
\end{CD}
\eeq
Such a pullback can be read as follows: the pullback of any map of $\RMod(C, \coprod_{\alpha}E_{\alpha})$ on which there is a $R$-module structure $r = \str_M$ is a map $P_{\alpha} \rarr E_{\alpha}$, element of some $R$-module on which there is a $R$-module structure $\str_{N_{\alpha}}$. The restriction of any map $C \rarr \coprod_{\alpha}E_{\alpha}$ of $\RMod(C, \coprod_{\alpha} E_{\alpha})$ along $B \rarr P_{\alpha} \rarr C$ factors through that composite, which by commutativity is in the image of $\str_{N_{\alpha}} \rarr \str_M$, where $M_{\alpha} = \RMod(C, E_{\alpha})$.\\

We now prove $a\phi$ is mono. Fix $C \in \Ob(\cC)$, $\phi_C$ the component of $\phi$ at $C$. Let $K$ be the kernel pair of $\phi$ in $\RModCop$. In $\RModCop$ such kernel pairs are computed pointwise so we have a kernel pair $K(C)$ of $\phi_C$ in $\RMod$. Using that coproducts are disjoint in $\cE$ one would show exactly as in App.3.4 that as a subobject of $(\coprod_{\alpha, \RModCop} \RMod(A, E_{\alpha}))^2$, $K$ is mapped by $\eta = i \circ a$ into the diagonal $\Delta \coprod_{\alpha} \RMod(A, E_{\alpha})$, which uses the fact that a sheaf of $R$-modules must take 0 to the trivial group. It follows that $aK \subset (\coprod_{\alpha} \RMod(A, E_{\alpha}))^2$ is contained in $\Delta \coprod_{\alpha} \RMod(A, E_{\alpha})$. $a$ being left exact, $aK$ is the kernel pair of $a\phi$, which implies that $a\phi$ is mono, hence $a \phi$ is an isomorphism, and thus $\RMod(A,-)$ preserves small coproducts. This completes the proof that $\RMod(A, -)$ preserves colimits.
\end{proof}

Finally $F \mapsto F \otimes_{\cC}A$ preserves colimits as a left adjoint, so does $\RMod(A,-)$, hence the composition $\eta_F: F \mapsto \RMod(A, F \otimes_{\cC}A)$ preserves colimits as well. Since any $F \in \ShRCJ$ is a colimit of representables, $\eta$ being an isomorphism on representables, it is an isomorphism in $\ShRCJ$ and we have proved that the unit of the Hom-tensor adjunction is an isomorphism, which completes the proof that $\cE \simeq \ShRCJ$.\\

Conversely we need to prove that $\ShRCJ$ is a category satisfying Giraud's conditions for being a Grothendieck topos, and is in addition such that $\Mor(\cE)$ is a strict assembly category. We have an embedding $i: \ShRCJ \rarr \RModCop$ with a left adjoint $a: \RModCop \rarr \ShRCJ$ that is left exact, so it preserves limits and colimits. Hence $\ShRCJ$ will satisfy (i)-(iv) if $\RModCop$ does. In $\RModCop$ limits and colimits are computed pointwise, so it suffices to show (i)-(iv) hold in $\RMod$. Sets has all small coproducts hence so does $\RMod$. Coproducts are disjoint in Sets, hence so are they in $\RMod$. The same holds for stability under pullbacks. Since exact forks are stable in Sets, they are stable in $\RMod$ as well. Indeed if we have:
\beq
\setlength{\unitlength}{0.5cm}
\begin{picture}(11,3)(0,0)
\thicklines
\put(0,1){$R$}
\put(5,1){$E$}
\put(10,1){$Q$}
\put(1,1){\vector(1,0){3}}
\put(1,1.5){\vector(1,0){3}}
\put(6,1){\vector(1,0){3}}
\put(2,2.2){$\str_M$}
\put(2,0){$\str_N$}
\put(7,2){$q$}
\end{picture}
\eeq
exact, $Q' \rarr Q$ any map, then:
\beq
R\times_Q Q' \rrarr E \times_Q Q' \xrarr{q \times 1} (Q \times_Q Q') \simeq Q'
\eeq
is exact as well.\\

Any epi in Sets is a coequalizer. Let $q:M \rarr N$ be an epi in $\RMod$ with underlying epi $\underline{q}: \underline{M} \rarr \underline{N}$ in Sets. Hence:
\beq
\setlength{\unitlength}{0.5cm}
\begin{picture}(12,8)(4,5)
\thicklines
\put(4.5,11){$\underline{P}$}
\put(9.5,11){$\underline{M}$}
\put(14.5,11){$\underline{N}$}
\put(4.5,6){$P$}
\put(9.5,6){$M$}
\put(14.5,6){$N$}
\put(6,11.5){\vector(1,0){3}}
\put(6,11){\vector(1,0){3}}
\put(7,12){$\underline{\partial_0}$}
\put(7,10){$\underline{\partial_1}$}
\put(6,6.5){\vector(1,0){3}}
\put(6,6){\vector(1,0){3}}
\put(7,7){$\partial_0$}
\put(7,5){$\partial_1$}
\put(11,11){\vector(1,0){3}}
\put(11,6){\vector(1,0){3}}
\put(13,12){$\underline{q}$}
\put(13,5){$q$}
\put(5,10){\vector(0,-1){3}}
\put(10,10){\vector(0,-1){3}}
\put(15,10){\vector(0,-1){3}}
\put(3,9){$\Str_P$}
\put(10.5,9){$\Str_M$}
\put(15.5,9){$\Str_N$}
\end{picture}
\eeq
where structural maps $\Str$ endow sets with a given structure of $R$-module, $q|_{\text{Sets}} = \underline{q}$ and $q(\str_M) = \str_N$, $\partial_i |_{\text{Sets}} = \underline{\partial}_i$ $i = 0,1$ and $\partial_i(\str_P) = \str_M$ for $i=0,1$. $q$ coequalizes $\partial_0$ and $\partial_1$ if $q \partial_0 = q \partial_1$ and is universal among such. On elements $\underline{q} \underline{\partial}_0 = \underline{q} \underline{\partial}_1$, and on algebraic structures $\partial_0 \str_P = \str_M = \partial_1 \str_P$. Universality: we have a unique map $\xi$ of sets as below.
\beq
\setlength{\unitlength}{0.5cm}
\begin{picture}(21,14)(4,2)
\thicklines
\put(4.5,11){$\underline{P}$}
\put(9.5,11){$\underline{M}$}
\put(14.5,11){$\underline{N}$}
\put(4.5,6){$P$}
\put(9.5,6){$M$}
\put(14.5,6){$N$}
\put(6,11.5){\vector(1,0){3}}
\put(6,11){\vector(1,0){3}}
\put(6,6.5){\vector(1,0){3}}
\put(6,6){\vector(1,0){3}}
\put(11,11){\vector(1,0){3}}
\put(11,6){\vector(1,0){3}}
\put(5,10){\vector(0,-1){3}}
\put(10,10){\vector(0,-1){3}}
\put(15,10){\vector(0,-1){3}}
\put(18.5,2.5){$R$}
\put(11,5){\vector(3,-1){7}}
\put(18.5,15){$\underline{R}$}
\put(11,12){\vector(3,1){6}}
\multiput(16,12)(0.4,0.4){6}{\circle*{0.15}}
\put(18,14){\vector(1,1){0.5}}
\put(17,12){$\xi$}
\multiput(16,6)(0.4,-0.4){6}{\circle*{0.15}}
\put(18,4){\vector(1,-1){0.5}}
\put(17,6){$\!$}
\put(19,14){\vector(0,-1){10}}
\put(20,9){$\Str_R$}
\end{picture}
\eeq
The lower map $N \rarr R$ is $\xi$ on elements and $\str_N \rarr \str_R$ on algebraic structures. To see that this latter map is unique we adopt the view that $R$-modules $M$ can be viewed as sets of selection objects, some of which select elements of the module, and one in particular selects the $R$-module structure $\str_M$ on $M$, that is $M \simeq \underline{M}^+$, hence for an exact fork of structural maps there is a unique structural map $\str_N \rarr \str_R$. Finally any equivalence relation $R \rrarr E$ in Sets is a kernel pair, hence so it is in $\RMod$. Indeed:
\beq
\setlength{\unitlength}{0.5cm}
\begin{picture}(15,19)(0,0)
\thicklines
\put(4.5,11){$\underline{P}$}
\put(9.5,11){$\underline{M}$}
\put(14.5,11){$\underline{N}$}
\put(4.5,6){$P$}
\put(9.5,6){$M$}
\put(14.5,6){$N$}
\put(6,11.5){\vector(1,0){3}}
\put(6,11){\vector(1,0){3}}
\put(6,6.5){\vector(1,0){3}}
\put(6,6){\vector(1,0){3}}
\put(11,11){\vector(1,0){3}}
\put(11,6){\vector(1,0){3}}
\put(5,10){\vector(0,-1){3}}
\put(10,10){\vector(0,-1){3}}
\put(15,10){\vector(0,-1){3}}
\put(3,9){$\Str_P$}
\put(10.5,9){$\Str_M$}
\put(15.5,9){$\Str_N$}
\put(3,18){$\underline{S}$}
\put(4,17.3){\vector(1,-1){4.9}}
\put(4,17.7){\vector(1,-1){5}}
\put(3.5,17){\vector(1,-4){1.2}}
\put(3,0){$S$}
\put(4,0.3){\vector(1,1){4.9}}
\put(4,0.7){\vector(1,1){5}}
\put(3.5,1){\vector(1,4){1.1}}
\qbezier(2.5,18)(0,18)(0,9)
\qbezier(0,9)(0,0)(2,0)
\put(2,0){\vector(1,0){0.5}}
\put(-2,9){$\Str_S$}
\end{picture}
\eeq
Under the structural maps $\Str$, the fact that equivalence relations are kernel pairs and have a quotient in Sets carries over to $\RMod$. Regarding $R$-module structures, we again regard those as selection objects, elements of a set characterizing $R$-modules. An equivalence relation has a quotient in Sets, hence it has one in $\RMod$ as well, and this would be shown following the same reasoning that epis are coequalizers. For (v), we have already proved that every sheaf in $\ShRCJ$ is a colimit of representables so it has a small set of generators. Thus $\ShRCJ$ is a $\RMod$-enriched category satisfying (i)-(v), since for $M,N$ two $R$-modules, $\Hom_R(M,N)$ is a $R$-module itself since $R$ is commutative. Since $\RMod$ corresponds to a strict assembly category by construction, $\ShRCJ$ is a Grothendieck topos enriched in a strict assembly category, and this proves the theorem.

\end{document}